\documentclass[a4paper, reqno, 11pt]{amsart}

\usepackage[UKenglish]{babel}
\usepackage{amsmath, amssymb, amsthm, amsfonts}
\usepackage{mathtools}
\usepackage{scrextend}
\usepackage[hidelinks]{hyperref}
\usepackage{xpatch}
\usepackage{color}
\usepackage{tikz-cd}
\usepackage[inline]{enumitem}
\usepackage{theoremref}
\usepackage{quiver}
\usepackage{ifthen}
\usepackage[textsize=tiny]{todonotes}

\usepackage[inline,final]{showlabels}
\showlabels{thlabel}

\xpatchcmd{\proof}{\itshape}{\normalfont\bfseries}{}{}
\newtheoremstyle{preview}{}{}{\itshape}{}{\bfseries}{.}{.5em}{#3, preview}

\newtheorem{theorem}{Theorem}[section]
\newtheorem{proposition}[theorem]{Proposition}
\newtheorem{lemma}[theorem]{Lemma}
\newtheorem{corollary}[theorem]{Corollary}
\newtheorem{fact}[theorem]{Fact}

\theoremstyle{definition}

\newtheorem{definition}[theorem]{Definition}
\newtheorem{remark}[theorem]{Remark}
\newtheorem{convention}[theorem]{Convention}

\newtheorem{question}[theorem]{Question}

\theoremstyle{preview}
\newtheorem*{preview-theorem}{Preview}

\newcommand{\A}{\mathcal{A}}

\newcommand{\C}{\mathcal{C}}
\newcommand{\D}{\mathcal{D}}

\newcommand{\K}{\mathcal{K}}
\renewcommand{\L}{\mathcal{L}}
\newcommand{\M}{\mathcal{M}}
\newcommand{\NN}{\mathfrak{N}}
\newcommand{\RR}{\mathcal{R}}
\newcommand{\X}{\mathcal{X}}
\newcommand{\Z}{\mathbb{Z}}

\newcommand{\Ab}{\mathbf{Ab}}
\newcommand{\Set}{\mathbf{Set}}
\newcommand{\Mod}{\mathbf{Mod}}
\newcommand{\Str}{\mathbf{Str}}
\newcommand{\Ind}{\mathbf{Ind}}

\newcommand{\Eff}[1][]{\ifthenelse{\equal{#1}{}}{\mathbf{Eff}}{#1\text{-}\mathbf{Eff}}}
\newcommand{\QuasiEff}[1][]{\ifthenelse{\equal{#1}{}}{\mathbf{QuasiEff}}{#1\text{-}\mathbf{QuasiEff}}}

\DeclareMathOperator{\colim}{colim}
\DeclareMathOperator{\dom}{dom}
\DeclareMathOperator{\cod}{cod}
\DeclareMathOperator{\Hom}{Hom}
\DeclareMathOperator{\rank}{rank}
\DeclareMathOperator{\Po}{Po}
\DeclareMathOperator{\Tc}{Tc}
\DeclareMathOperator{\Rt}{Rt}
\DeclareMathOperator{\cell}{cell}
\DeclareMathOperator{\cof}{cof}
\DeclareMathOperator{\cf}{cf}
\DeclareMathOperator{\Obj}{Obj}
\DeclareMathOperator{\Mor}{Mor}
\DeclareMathOperator{\Term}{Term}

\renewcommand{\phi}{\varphi}

\newcommand{\op}{{\textup{op}}}

\title{Cellular generation revisited}
\author{Sean Cox, Mark Kamsma, and Ji\v{r}\'{i} Rosick\'{y}}

\date{\today. \emph{MSC2020}: 03E75 (primary), 18C35}

\address[Sean Cox]{Department of Mathematics and Applied Mathematics\newline
Virginia Commonwealth University\newline
1015 Floyd Ave, VA 23284, Richmond, USA}
\email[]{scox9@vcu.edu}
\urladdr{https://skolemhull.wordpress.com/}

\address[Mark Kamsma]{Department of Mathematics and Statistics\newline
Masaryk University, Faculty of Sciences\newline
Kotl\'{a}\v{r}sk\'{a} 2, 611 37, Brno, Czech Republic}
\email[]{mark@markkamsma.nl}
\urladdr{https://markkamsma.nl}

\address[Ji\v{r}\'{i} Rosick\'{y}]{Department of Mathematics and Statistics\newline
Masaryk University, Faculty of Sciences\newline
Kotl\'{a}\v{r}sk\'{a} 2, 611 37, Brno, Czech Republic}
\email[]{rosicky@math.muni.cz}
\urladdr{https://www.math.muni.cz/~rosicky/}

\thanks{\begin{minipage}{0.6\textwidth}
The first author is supported by NSF grant DMS-2154141. The second author is supported by Marie Sk\l{}odowska-Curie grant number 101130801.
\end{minipage}%
\begin{minipage}{0.4\textwidth}
\begin{center}
    \includegraphics[height=1.1cm]{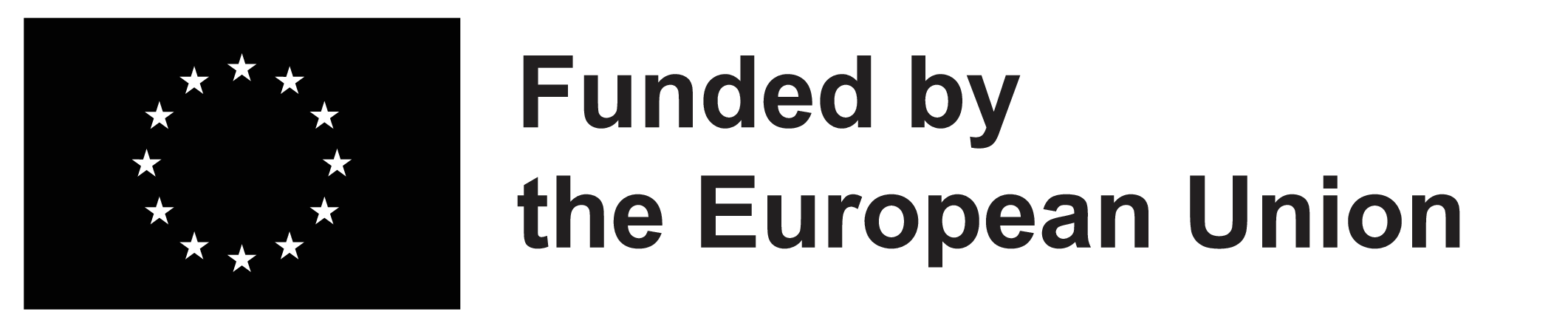}
\end{center}
\end{minipage}
\nopunct}
 
\begin{document}

\begin{abstract}
Cellular generation, which generalises cofibrant generation, is an important categorical smallness condition on a class of morphisms. A general challenge is to determine whether a given class of morphisms $\M$ is cellularly generated, in which \emph{$\M$-effective squares} are often useful. These are commuting squares consisting of morphisms in $\M$, so that the induced morphism from the pushout square is also in $\M$. When we drop the requirement that the vertical morphisms in the square are in $\M$ we obtain the weaker notion of \emph{$\M$-quasieffective square}. We prove that, in a locally presentable category, $\M$ is cellularly generated if and only if $\M$ is almost everywhere quasieffective. The latter is a set-theoretic condition stating that for almost every partial elementary set-theoretic subuniverse $\NN$, we have that restricting any morphism in $\M$ to $\NN$ yields an $\M$-quasieffective square. For locally finitely presentable categories this yields an additional categorical characterisation in terms of filtrations of $\M$-quasieffective squares.

If we additionally assume that $\M$ is continuous (i.e., the corresponding wide subcategory is closed under directed colimits) then we obtain a stronger characterisation of cellular generation in terms of accessibility of the category of $\M$-effective squares. This improves on a theorem by Lieberman, Vasey, and the third author.
\end{abstract}

\maketitle

\tableofcontents

\section{Introduction}
Smallness conditions have played an important role in category theory since its beginning. For instance, Freyd's \emph{solution set condition} guarantees that a limit preserving functor on a complete category has a left adjoint. Another example is Isbell's important concept of a small \emph{dense} subcategory (he called it left adequate). Both concepts can be found in \cite{maclaneCategoriesWorkingMathematician2013}---and the latter entered the definition of a \emph{locally presentable category} of Gabriel and Ulmer \cite{gabrielLokalPraesentierbareKategorien1971} or the more general concept of an \emph{accessible category} of Makkai and Par\'e \cite{makkaiAccessibleCategoriesFoundations1989a}. Recall that a category is locally $\lambda$-presentable if it is cocomplete and has a small dense subcategory formed by $\lambda$-presentable objects ($\lambda$ is a regular cardinal), while in $\lambda$-accessible categories cocompleteness reduces to the existence of $\lambda$-directed colimits. Additionally, the density reduces to the fact that every object is a $\lambda$-directed colimit of $\lambda$-presentable objects, but the requirement that there is only a set of them remains. Locally presentable and accessible categories pervaded many areas, including algebra, homotopy theory, and model theory. 

In homotopy theory, there is another smallness condition due to Quillen \cite{quillenHomotopicalAlgebra1967} called a \emph{small object argument}. Its basis is that a given class $\M$ of morphisms is generated from a set using pushouts, transfinite compositions, and retracts.  In this case, $\M$ is called \emph{cofibrantly generated}, or \emph{cellularly generated} if retracts are not used. In homotopy theory, if both cofibrations and trivial cofibrations are cofibrantly generated then they yield a model category which is called cofibrantly generated \cite{hoveyModelCategories2014}. A general challenge is to determine whether a given class of morphisms is cofibrantly generated. In 2000, Hovey asked for examples of model categories which are not cofibrantly generated. Soon after, such examples appeared---for instance in \cite{adamekWeakFactorizationSystems2002} based on the fact that embeddings in posets are not cofibrantly generated. 

In model theory, specifically in stability theory \cite{shelahClassificationTheoryNumber1990}, another smallness condition appears, often referred to as \emph{local character}. A central role in stability theory is played by independence relations, generalising for example linear independence in vector spaces and algebraic independence in algebraically closed fields. Local character means that, given subsets $A$ and $B$ of some structure such that $A$ is small (e.g., finite), then $A$ can only depend on a small part of $B$. In \cite{liebermanForkingIndependenceCategorical2019} a categorical approach to model-theoretic independence was started. There, the independence relation itself can be viewed as a category, and local character can be expressed as the accessibility of that category.

In \cite{barrCategoriesEffectiveUnions1988} Barr defined the notion of a category having \emph{effective unions}. This means that for any pullback square that consists of regular monomorphisms, like below on the left, the induced morphism from the relevant pushout is a regular monomorphism, as pictured below on the right.
\[\begin{tikzcd}[sep=small]
	A && B &&&& B \\
	&&&& A & P \\
	C && D && C & D
	\arrow[from=1-1, to=1-3]
	\arrow[curve={height=-6pt}, from=2-5, to=1-7]
	\arrow[from=2-5, to=2-6]
	\arrow[dashed, from=2-6, to=1-7]
	\arrow["\lrcorner"{anchor=center, pos=0.125, rotate=-90}, draw=none, from=2-6, to=3-5]
	\arrow[from=3-1, to=1-1]
	\arrow[from=3-1, to=3-3]
	\arrow[from=3-3, to=1-3]
	\arrow[from=3-5, to=2-5]
	\arrow[from=3-5, to=3-6]
	\arrow[curve={height=6pt}, from=3-6, to=1-7]
	\arrow[from=3-6, to=2-6]
\end{tikzcd}\]
Borceux and the third author showed that this condition, with regular monomorphisms replaced by pure monomorphisms, implies cofibrant generation of pure monomorphisms \cite{borceuxPurityAlgebra2007}. However, cofibrant generation is not equivalent to having effective unions (e.g., \cite[Example 2.5]{liebermanCofibrantGenerationPure2020}). Still, for a fixed class of morphisms $\M$, the idea of \emph{$\M$-effective squares} is useful in characterising cofibrant generation and cellular generation. By an $\M$-effective square we now mean a commuting square whose arrows are all in $\M$ and where the induced morphism from the relevant pushout is in $\M$. In both \cite{liebermanForkingIndependenceCategorical2019} and \cite{liebermanCellularCategoriesStable2023} these were shown to induce a model-theoretic independence relation, and \cite[Theorem 3.1]{liebermanCellularCategoriesStable2023} shows that a nicely behaved class $\M$ in a locally presentable category $\K$ is cofibrantly generated if and only if the induced independence relation, viewed as a category, is accessible. Hence, the result relates two seemingly disparate smallness conditions, one coming from homotopy theory and the other from model theory. By ``nicely behaved'' we mean that $\M$ is cofibrantly closed (i.e., closed under pushouts, transfinite compositions, and retracts), continuous (i.e., the wide subcategory $\K_\M$ of $\K$ having $\M$ as morphisms is closed under directed colimits in $\K$), and coherent (i.e., if $g, gf \in \M$ then $f \in \M$).

The current paper further elaborates this equivalence. We fix a gap in the proof of \cite[Theorem 3.1]{liebermanCellularCategoriesStable2023}. This fix is not evident and uses a \emph{fat small object argument} of \cite{makkaiFatSmallObject2014} which is a variation of a small object argument. Furthermore, we remove the assumptions of continuity and coherence of $\M$ at the expense of replacing $\M$-effective squares by a weaker notion of \emph{$\M$-quasieffective} squares, where only horizontal morphisms are in $\M$. Also, we have to replace accessibility of the resulting category by the existence of filtrations of objects by smaller objects. The proof combines categorical and set-theoretical methods. The set-theoretical tools that we rely on were developed by the first author, and are based on partially elementary subuniverses of the set-theoretic universe. They were used to prove the flat cover conjecture for certain categories of monoid acts \cite{coxFlatCoverConjecture2025}.

\textbf{Main results.} We discuss, and state, the main results of our paper. The statements are precise, although not all components are defined yet and the actual statements contain more details.

We recall that a \emph{filtration} of an object $A$ in a category is a chain of smaller objects such that $A$ is the colimit of this chain. Being ``smaller'' is made precise using the language of presentability of objects, which is a categorical way of measuring their size. Our main result is the following.
\begin{preview-theorem}[\thref{cellular-generation-lfp}]
Let $\K$ be a locally finitely presentable category and let $\M$ be a cellularly closed class of morphisms in $\K$. Then the following are equivalent.
\begin{enumerate}[label=(\roman*)]
\item The class $\M$ is cellularly generated.
\item There is a cardinal $\mu$ such that every $f \in \M$, with $\rank_{\K^2}(f) > \mu$, admits a smooth filtration in $\K^2$ that consists of $\M$-quasieffective squares.
\item The class $\M$ is almost everywhere quasieffective.
\end{enumerate}
\end{preview-theorem}
The condition of being almost everywhere quasieffective is a set-theoretic condition. It roughly means the following. Given a partially elementary subuniverse $\NN$ of the set-theoretic universe $V$ and a morphism $f: A \to B$ in $\mathfrak{N} \cap \M$, we can make sense of the restriction $f \restriction \NN$ of $f$ to $\NN$, which boils down to viewing $A$ and $B$ as some concrete objects (e.g., as models of some first-order theory) and intersecting their underlying sets with $\NN$. We can then form the square below and ask whether it is $\M$-quasieffective.
\[\begin{tikzcd}[sep=small]
	A & B \\
	{A \restriction \NN} & {B \restriction \NN}
	\arrow["f", from=1-1, to=1-2]
	\arrow[hook, from=2-1, to=1-1]
	\arrow["{f \restriction \NN}"', from=2-1, to=2-2]
	\arrow[hook, from=2-2, to=1-2]
\end{tikzcd}\]
We then say that $\M$ is almost everywhere quasieffective if this holds for ``almost all" $\mathfrak{N}$, i.e., there is a set of parameters $p$, such that for any partially elementary subuniverse $\NN$ that contains $p$ and for any $f \in \mathfrak{N} \cap \M$ the square above is $\M$-quasieffective.

In the generality of locally presentable categories, we obtain a similar theorem, but we have to give up the condition about filtrations (see also \thref{locally-presentable-filtrations} and the preceding discussion).
\begin{preview-theorem}[\thref{cellular-generation-any-lp}]
Let $\K$ be a locally $\lambda$-presentable category and let $\M$ be a cellularly closed class of morphisms in $\K$. Then the following are equivalent.
\begin{enumerate}[label=(\roman*)]
\item The class $\M$ is cellularly generated.
\item The class $\M$ is almost everywhere quasieffective.
\end{enumerate}
\end{preview-theorem}
Finally, we fix the gap in \cite[Theorem 3.1]{liebermanCellularCategoriesStable2023}, see \thref{the-gap} for details. We also improve on \cite[Theorem 3.1]{liebermanCellularCategoriesStable2023}, which is about cofibrant generation, by providing a version for cellular generation.
\begin{preview-theorem}[\thref{cofibrant-generation-continuous}]
Let $\K$ be a locally presentable category and let $\M$ be a coherent continuous cofibrantly closed class of morphisms in $\K$. Then the following are equivalent.
\begin{enumerate}[label=(\roman*)]
\item The class $\M$ is cofibrantly generated and $\K_\M$ is accessible (with directed colimits).
\item The category $\Eff[\M](\K)$ is accessible (with directed colimits).
\end{enumerate}
\end{preview-theorem}
\begin{preview-theorem}[\thref{cellular-generation-continuous}]
Let $\K$ be a locally presentable category and let $\M$ be a coherent continuous cellularly closed class of monomorphisms in $\K$. Then the following are equivalent.
\begin{enumerate}[label=(\roman*)]
\item The class $\M$ is cellularly generated and $\K_\M$ is accessible (with directed colimits).
\item The category $\Eff[\M](\K)$ is accessible (with directed colimits)
\end{enumerate}
\end{preview-theorem}
Here $\Eff[\M](\K)$ is the category whose objects are morphisms in $\M$ and whose morphisms are $\M$-effective squares. The proofs of these theorems are purely categorical and are thus separate from the first two main results.

Finally, we note that our theorems are about cellular generation, instead of cofibrant generation. This is more general, because a class of morphisms is cofibrantly generated if and only if it is cellularly generated and is closed under retracts (see \thref{elimination-of-retracts}).

\textbf{Overview.} We discuss the preliminaries in Section \ref{sec_SetTheorySetup}, split up in categorical preliminaries (Section \ref{sec:categorical-preliminaries}) and set-theoretical preliminaries (Section \ref{sec_SetTheorySetup}). In Section \ref{sec:cellular-generation-lfp} we state and prove our main theorem: the characterisation of cellular generation in locally finitely presentable categories. We provide a full statement of the theorem as early as possible in the section, after which we prove its components individually in multiple lemmas and propositions. The statements of the individual components are a bit more detailed, for example with more information about the parameters involved or a slightly more general setup because a certain assumption is not used for that particular component. We take the same approach in the remaining sections. Section \ref{sec_AnyLocallyPresent} provides a version of our main theorem for all locally presentable categories. It makes essential use of our main theorem, as the proof strategy is to present the arbitrary locally presentable category $\K$ as a full reflective subcategory of a presheaf category. Since presheaf categories are always locally finitely presentable, we can then apply our main theorem to the presheaf category, and the transfer the characterisation via the reflection functor. In Section \ref{sec:cellular-generation-for-continuous-classes-of-morphisms} we provide the improved version of \cite[Theorem 3.1]{liebermanCellularCategoriesStable2023}. The proofs in this section are purely categorical. 

In Appendix \ref{sec_LevyComplexity} we provide some computations of the L\'evy complexities of certain statements that we use throughout the paper.

\section{Preliminaries}
\label{sec:preliminaries}

\subsection{Categorical preliminaries}
\label{sec:categorical-preliminaries}
We work in the context of locally presentable and accessible categories. We recall the basic definitions and refer to \cite{adamekLocallyPresentableAccessible1994} for further reference.
\begin{definition}
\thlabel{presentability}
Let $\lambda$ be a regular cardinal. An object $A$ in a category $\K$ is called \emph{$\lambda$-presentable} if $\Hom(A, -): \K \to \Set$ preserves $\lambda$-directed colimits. If $\lambda = \omega$ we will also say that $A$ is \emph{finitely presentable}.

For any cardinal $\kappa$ we call an object \emph{$(< \kappa)$-presentable} if it is $\lambda$-presentable for some $\lambda < \kappa$.
\end{definition}
Intuitively, $A$ being $\lambda$-presentable means that it ``has size $<\lambda$'', see for example \thref{limit-theory-standard-facts}(iii). Unfolding \thref{presentability} yields the following more practical formulation, which we will often use. An object $A$ is $\lambda$-presentable if any morphism $f: A \to B$ into a $\lambda$-directed colimit $B = \colim_{i \in I} B_i$ factors essentially uniquely through the diagram $(B_i)_{i \in I}$. That is, $f$ factors as $A \xrightarrow{f'} B_i \to B$ for some $i \in I$ and this factorisation is \emph{essentially unique}, meaning that for any other factorisation $A \xrightarrow{f''} B_i \to B$ of $f$ there is $j \in I$ such that $A \xrightarrow{f'} B_i \to B_j$ is the same as $A \xrightarrow{f''} B_i \to B_j$.
\begin{definition}
\thlabel{locally-presentable-accessible-category}
Let $\lambda$ be a regular cardinal. A category $\K$ is called \emph{$\lambda$-accessible} if:
\begin{enumerate}
\item it has $\lambda$-directed colimits,
\item there is a set $\A$ of $\lambda$-presentable objects such that every object in $\K$ is a $\lambda$-directed colimit of objects in $\A$.
\end{enumerate}
We call $\K$ \emph{accessible} if it is $\lambda$-accessible for some $\lambda$.

A category is called \emph{locally ($\lambda$-)presentable} if it is ($\lambda$-)accessible and it is cocomplete. As before, if $\lambda = \omega$ we speak of \emph{finitely accessible} and \emph{locally finitely presentable} categories instead.
\end{definition}
Locally presentable categories are also complete, but this is a consequence of the definition and we will have no use for this fact. What is important to us is that locally presentable categories can generally be described as the category of models of a certain type of theory.
\begin{definition}
\thlabel{limit-theory}
Fix some language $\L$. A \emph{limit sentence} in $\L$ is one of the form
\[
\forall \bar{x} (\phi(\bar{x}) \to \exists!\bar{y} \psi(\bar{x}, \bar{y})),
\]
where $\phi(\bar{x})$ and $\psi(\bar{x}, \bar{y})$ are conjunctions of atomic formulas (or $\top$, i.e., ``truth''). A \emph{limit theory} $T$ is a set of limit sentences. We write $\Mod(T)$ for the category of models of $T$, and the morphisms are homomorphisms of $\L$-structures.
\end{definition}
Languages may be multi-sorted, which requires no special treatment. When considering languages, logical formulas, and theories we stick to finitary logic (unless explicitly mentioned otherwise). While infinitary logic has a rich interaction with locally presentable categories, we will only be interested in the finitary case.
\begin{fact}[{\cite[Theorem 5.9]{adamekLocallyPresentableAccessible1994}}]
\thlabel{lfp-iff-models-of-limit-theory}
A category is locally finitely presentable if and only if it is equivalent to the category of models $\Mod(T)$ of a limit theory $T$.
\end{fact}
We will often use the following standard facts implicitly.
\begin{fact}
\thlabel{limit-theory-standard-facts}
Let $T$ be a limit theory in a language $\L$.
\begin{enumerate}[label=(\roman*)]
\item Any substructure of a model of $T$ is a model of $T$.
\item Directed colimits in $\Mod(T)$ are computed the same as in the category of $\L$-structures.
\item Let $\lambda > |\L| + \aleph_0$ be a regular cardinal. Then a model $M$ of $T$ is $\lambda$-presentable in $\Mod(T)$ if and only if $|M| < \lambda$.
\end{enumerate}
\end{fact}
\begin{proof}
Item (i) follows because a limit theory is in particular universally axiomatised. Item (ii) is verified in the proof of \cite[Theorem 5.9]{adamekLocallyPresentableAccessible1994}. We prove (iii).

By our assumption on $\lambda$, the free $\L$-structure on $< \lambda$ generators has cardinality $< \lambda$. So for any $A \subseteq M$ with $|A| < \lambda$, the substructure $\langle A \rangle \subseteq M$ generated by $A$ satisfies $|\langle A \rangle| < \lambda$. So $M$ is the $\lambda$-directed colimit in $\Mod(T)$ of all its substructures of cardinality $< \lambda$, see also (i). So if $M$ is $\lambda$-presentable then the identity on $M$ must factor through this diagram, and so $|M| < \lambda$. The converse is standard.
\end{proof}

\begin{definition}
\thlabel{presentability-rank}
The \emph{presentability rank} of an object $A$ in an accessible category $\K$ is the smallest regular cardinal $\lambda$ such that $A$ is $\lambda$-presentable. We write $\rank_\K(A)$ for the presentability rank of $A$ in $\K$.
\end{definition}
\begin{fact}[{\cite[Lemma 4.2]{bekeAbstractElementaryClasses2012}}]
\thlabel{successor-presentability-rank}
In a $\lambda$-accessible category with directed colimits the presentability rank of any object that is not $\lambda$-presentable is a successor cardinal.
\end{fact}

\begin{definition}
\thlabel{chains-and-filtrations}
We call an ordinal indexed diagram $(A_i)_{i < \delta}$ a \emph{chain}. We call such a chain \emph{smooth} if $A_\ell = \colim_{i < \ell} A_i$ for all limit ordinals $\ell < \delta$.

We call a smooth chain $(A_i)_{i < \delta}$ a \emph{smooth $(< \kappa)$-small chain} if $\delta < \kappa$ and $A_i$ is $(< \kappa)$-presentable for all $i < \delta$. A \emph{smooth $\kappa$-small chain} is a smooth $(< \kappa^+)$-small chain, i.e., $\delta \leq \kappa$ and $A_i$ is $\kappa$-presentable for all $i < \delta$.

Let $A$ be an object in some accessible category $\K$ and write $\lambda := \rank_\K(A)$. A \emph{smooth filtration of $A$} is a smooth $(< \lambda)$-small chain with colimit $A$. We say that \emph{$A$ admits a smooth filtration} if there is a smooth filtration of $A$.
\end{definition}

\begin{definition}
\thlabel{composable-class}
Let $\M$ be a class of morphisms in a category $\K$.
\begin{itemize}
\item If $\M$ is closed under composition and contains all isomorphisms we call it \emph{composable}.
\item If $gf \in \M$ implies that $f \in \M$ we call it \emph{left-cancellable}.
\item If $gf,g \in \M$ implies that $f \in \M$ we call it \emph{coherent}.
\end{itemize}
Given a composable class of morphisms $\M$ we write $\K_\M$ for the wide subcategory of $\K$ on the morphisms in $\M$. That is, $\K_\M$ has the same objects as $\K$, but only the morphisms in $\M$.
\end{definition}
We collect some useful facts that may be considered standard knowledge, but we recall them here for the reader's convenience. Recall that $\K^2$ denotes the category of morphisms of $\K$. That is, its objects are morphisms in $\K$ and its morphisms are commutative squares.
\begin{fact}
\thlabel{presentability-facts}
Let $\K$ be any category.
\begin{enumerate}[label=(\roman*)]
\item The colimit (if it exists) of a $\lambda$-small diagram (i.e., the diagram has $< \lambda$ many objects and morphisms) that consists of $\lambda$-presentable objects is again $\lambda$-presentable.
\item Colimits in $\K^2$ are computed componentwise. That is, for a diagram $(f_i: A_i \to B_i)$ in $\K^2$ such that $K = \colim_{i \in I} A_i$ and $L = \colim_{i \in I} B_i$ exist in $\K$, the induced morphism $f: A \to B$ is the colimit in $\K^2$.
\end{enumerate}
\end{fact}
\begin{proof}
Item (i) is \cite[Proposition 1.16]{adamekLocallyPresentableAccessible1994}, and item (ii) is well known.
\end{proof}
\begin{definition}
\thlabel{m-sub-lambda}
Let $\M$ be a class of morphisms in a category $\K$. For any cardinal $\lambda$ we write $\M_\lambda$ and $\M_{< \lambda}$ for the classes of morphisms in $\M$ that are, respectively, $(< \lambda^+)$-presentable and $(< \lambda)$-presentable in $\K^2$. Here $\K^2$ is the category of morphisms of $\K$. That is, its objects are morphisms in $\K$ and its morphisms are commutative squares.
\end{definition}
If $\lambda$ is a regular cardinal then $M_\lambda$ simplifies to the class of $\lambda$-presentable morphisms. The reason for the above formulation is that it also makes sense for singular cardinals. The following fact shows that the presentability of a morphism can easily be computed in terms of the presentability of its domain and codomain.
\begin{lemma}
\thlabel{presentability-of-arrows}
If a category $\K$ has the amalgamation property (i.e., any span of morphisms can be completed to a commuting square) then a morphism $f: A \to B$ is $\lambda$-presentable as an object of $\K^2$ if and only if $A$ and $B$ are both $\lambda$-presentable as objects of $\K$.
\end{lemma}
\begin{proof}
The right to left direction is well known and straightforward, and does not use the amalgamation property. We prove the other direction.

To show that $B$ is $\lambda$-presentable we actually do not need the amalgamation property. Let $(d_i: D_i \to D)_{i \in I}$ be a $\lambda$-directed colimit in $\K$, and let $g: B \to D$ be some morphism. Then $((d_i, d_i): id_{D_i} \to id_D)_{i \in I}$ is a $\lambda$-directed colimit in $\K^2$. We also have a morphism $(gf, g): f \to id_D$ as pictured in the diagram below.
\[\begin{tikzcd}[sep=small]
	D && D & \\
	& {D_i} && {D_i} \\
	A && B
	\arrow[equals, from=1-1, to=1-3]
	\arrow["{d_i}"', from=2-2, to=1-1]
	\arrow[equals, from=2-2, to=2-4]
	\arrow["{d_i}"', from=2-4, to=1-3]
	\arrow["gf", from=3-1, to=1-1]
	\arrow[dashed, from=3-1, to=2-2]
	\arrow["f"', from=3-1, to=3-3]
	\arrow["g"{pos=0.2}, from=3-3, to=1-3]
	\arrow[dashed, from=3-3, to=2-4]
\end{tikzcd}\]
Since $f$ was assumed to be $\lambda$-presentable, $(gf, g)$ factors essentially uniquely through some $(d_i, d_i)$, as depicted by the dashed arrows in the diagram above. So $g$ factors essentially uniquely through some $d_i$.

To show that $A$ is $\lambda$-presentable we will use the amalgamation property. Let again $(d_i: D_i \to D)_{i \in I}$ be a $\lambda$-directed colimit in $\K$, and this time we consider an arbitrary morphism $g: A \to D$. Using the amalgamation property we can complete the span $(f, g)$ to a commuting diagram as shown below.
\[\begin{tikzcd}[sep=small]
	D && E & \\
	& {D_i} && E \\
	A && B
	\arrow["e", from=1-1, to=1-3]
	\arrow["{d_i}"', from=2-2, to=1-1]
	\arrow["{e d_i}"{pos=0.2}, from=2-2, to=2-4]
	\arrow[equals, from=2-4, to=1-3]
	\arrow["g", from=3-1, to=1-1]
	\arrow[dashed, from=3-1, to=2-2]
	\arrow["f"', from=3-1, to=3-3]
	\arrow["{g'}"{pos=0.2}, from=3-3, to=1-3]
	\arrow[dashed, from=3-3, to=2-4]
\end{tikzcd}\]
Then $e: D \to E$ is the $\lambda$-directed colimit of $(e d_i: D_i \to E)_{i \in I}$ in $\K^2$. So since $f$ is $\lambda$-presentable we have that $(g, g')$ factors essentially uniquely through some $(d_i, id_E)$, as depicted by the dashed arrows in the diagram above. So $g$ factors essentially uniquely through $d_i$.
\end{proof}
\begin{definition}
\thlabel{cellular-cofibrant-closure}
Let $\M$ be a class of morphisms in some cocomplete category $\K$.
\begin{itemize}
\item We write $\Po(\M)$ for the closure of $\M$ under pushouts (in $\K$). That is, $f \in \Po(\M)$ iff $f$ is an isomorphism or if there is a pushout diagram in $\K$ like below with $g \in \M$.
\[\begin{tikzcd}[sep=small]
	A & B \\
	C & D
	\arrow["f", from=1-1, to=1-2]
	\arrow["\lrcorner"{anchor=center, pos=0.125, rotate=-90}, draw=none, from=1-2, to=2-1]
	\arrow[from=2-1, to=1-1]
	\arrow["g"', from=2-1, to=2-2]
	\arrow[from=2-2, to=1-2]
\end{tikzcd}\]
\item We write $\Tc(\M)$ for the closure of $\M$ under transfinite compositions (in $\K$). That is, $f \in \Tc(\M)$ iff there is a smooth chain $(A_i)_{i \leq \delta}$ in $\K$ such that for all $i < \delta$ the link $A_i \to A_{i+1}$ is in $\M$ and $f$ is $A_0 \to A_\delta$.
\item We write $\Rt(\M)$ for the closure of $\M$ under retracts (in $\K^2$). That is, $f \in \Rt(\M)$ iff $f$ is a retract in $\K^2$ of some $g \in \M$.
\item We write $\cell(\M) := \Tc(\Po(\M))$ and $\cof(\M) := \Rt(\cell(\M))$.
\end{itemize}
The class $\M$ is called \emph{cellularly closed} if $\M = \cell(\M)$. It is called \emph{cofibrantly closed} if $\M = \cof(\M)$.

If there is a set $\M_0$ of morphisms such that $\M = \cell(\M_0)$ then $\M$ is called \emph{cellularly generated}. Similarly, if there is a set of morphisms $\M_0$ such that $\M = \cof(\M_0)$ then $\M$ is called \emph{cofibrantly generated}.
\end{definition}
All the operations in \thref{cellular-cofibrant-closure} are idempotent. Furthermore, any cofibrantly closed class is cellularly closed \cite[Lemma 2.2]{makkaiCellularCategories2014}.

Not every cellularly generated class of morphisms is cofibrantly generated, as it may not be cofibrantly closed. The following fact shows that the converse does hold: cofibrantly generated classes of morphisms are cellularly generated.
\begin{fact}[{\cite[Theorem B.1]{makkaiFatSmallObject2014}}]
\thlabel{elimination-of-retracts}
Let $\lambda$ be an uncountable regular cardinal, $\K$ a locally $\lambda$-presentable category, and $\M$ a cofibrantly generated class of morphisms in $\K$ with $\M = \cof(\M_\lambda)$. Then $\M = \cell(\M_\lambda)$.
\end{fact}

\subsection{Set-theoretical preliminaries}
\label{sec_SetTheorySetup}

Because we will exclusively work with concrete presentations of locally presentable categories given by \cite{adamekLocallyPresentableAccessible1994}, it will suffice to take as our background theory the Zermelo–Fraenkel axioms of set theory with choice (ZFC).

The L\'evy hierarchy provides an important stratification of formulas in the language of set theory (which has the single binary relation symbol $\in$).  A bounded quantifier is one of the form $\forall x \in y$ or $\exists x \in y$, and a formula is in the class $\Delta_0$ if it can be written as a block of bounded quantifiers followed by a quantifier-free formula.  The key feature of $\Delta_0$-formulas is that they are absolute between \emph{transitive} models of set theory \cite[Lemma 12.9]{MR1940513}.  In particular, if $H$ is a transitive set (i.e., $a \in b \in H \ \implies a \in H$) and $\varphi$ is a $\Delta_0$-formula, then for all $a_1,\dots,a_k \in H$ we have $(H,\in) \models \varphi(a_1,\dots,a_k)$ if and only if $(V,\in) \models \varphi(a_1,\dots,a_k)$.  So $(H,\in)$ computes the truth value of $\Delta_0$-statements correctly.  The L\'evy classes $\Sigma_n$ and $\Pi_n$ of formulas are defined recursively:  $\Sigma_0 = \Pi_0 = \Delta_0$; a $\Sigma_{n+1}$-formula is one of the form 
\[
\exists x_1 \dots \exists x_\ell \ \varphi(x_1,\dots,x_\ell,v_1,\dots,v_k)
\]
where $\varphi$ is a $\Pi_n$-formula, and a $\Pi_{n+1}$-formula is one of the form
\[
\forall x_1 \dots \forall x_\ell \ \varphi(x_1,\dots,x_\ell,v_1,\dots,v_k)
\]
where $\varphi$ is $\Sigma_n$. Every first-order formula lies in this hierarchy.

The absoluteness of $\Delta_0$-formulas between transitive sets implies the following fact.
\begin{fact}
\thlabel{UpwardAbsolutenessSigma_1}
The $\Sigma_1$-formulas are upward absolute, and the $\Pi_1$-formulas are downward absolute, between transitive models.    
\end{fact}
A (possibly proper) class $C$ is $\Sigma_n$-definable in $(V,\in)$ from parameters $p_1,\dots,p_k$ if there is a $\Sigma_n$-formula $\varphi$ such that $C = \{ x \ : \ (V,\in) \models \varphi(x,p_1,\dots,p_k)   \}$, and similarly for $\Pi_n$. 

For a set $N$, we write $\NN$ for the structure $(N,\in \cap  (N \times N))$, and we will often confuse $N$ and $\NN$. The next fact is a basic application of the L\'evy-Montague Reflection Principle \cite[Theorem 12.14]{MR1940513}, the first-order expressibility (for \emph{fixed} $n$) of $\Sigma_n$ truth in $(V,\in)$ \cite[Chapter 0]{MR2731169}, and the downward L\"owenheim-Skolem Theorem from first order logic.
\begin{fact}{\cite[Fact 3.5]{coxFlatCoverConjecture2025}}
\thlabel{Class_LS_Theorem}
For any fixed (metamathematical) natural number $n$ and any regular uncountable cardinal $\kappa$:  for every set $X$ of size $<\kappa$,  there is a set $N$ of size $<\kappa$ such that $X \subset N$, $N \cap \kappa$ is transitive,\footnote{This means $N \cap \kappa$ is downward closed as a set of ordinals.} and $\NN \prec_n (V,\in)$. The latter means that for every $\Sigma_n$-formula $\varphi$ and all $a_1,\dots,a_k \in N$ we have $\NN \models \varphi(a_1,\dots,a_k)$ if and only if $(V,\in) \models \varphi(a_1,\dots,a_k)$.
\end{fact}
\begin{fact}
\thlabel{kappa-sized-is-subset-of-subuniverse}
Let $\kappa$ be an infinite cardinal and suppose that $\NN \prec_1 (V, \in)$ and $\NN \cap \kappa$ is transitive. Then for any $A \in \NN$ we have that if $|A| < \kappa$ then $A \subseteq \NN$.
\end{fact}
\begin{proof}
This fact is well known, but it admits a short proof, so we provide it here. Suppose $A \in \NN$ and $|A|<\kappa$; then $(V,\in)$ satisfies the $\Sigma_1$-statement ``$\exists f$ $\exists \alpha$ $\alpha$ is an ordinal and $f$ is a surjection from $\alpha$ to $A$" (the part after the quantifiers is $\Delta_0$-expressible, see \cite[Chapter 12]{MR1940513}).  Since $A \in \NN \prec_1 (V,\in)$, there exists such a pair $\alpha,f$ that are both elements of $\NN$.  Since $|A|<\kappa$, it follows that $\alpha < \kappa$, hence $\alpha \in \NN \cap \kappa$.  Since $\NN \cap \kappa$ is transitive, $\alpha = \text{dom}(f) \subset \NN$.  It follows that $A=\text{im}(f)$ is also contained in $\NN$, since for any $\xi \in \text{dom}(f)$ the assertion $\exists a \in A \ \langle \xi,a \rangle \in f$ is a $\Delta_0$-statement in parameters $\xi$, $f$, and $A$, all of which are elements of $\NN$.
\end{proof}
If $\NN \prec_0 (V,\in)$ and $\L$ is a finitary sorted language with set $S$ of sorts, we will say that $\L$ ``is both an element and subset of $\NN$", or write $\mathcal{L} \cup \{ \L \} \subset \NN$, to mean that both $S$ and $\mathcal{L}$ are elements of $\NN$, $S$ is a subset of $\NN$, and each $\mathcal{L}$-symbol (with its associated arity) is an element of $\NN$.  We view an $\mathcal{L}$-structure as an $S$-partitioned set $A = \bigsqcup_{s \in S} A_s$ together with a function that, for each symbol in $\L$, assigns its interpretation.  We often confuse the $\mathcal{L}$-structure with its underlying set.  Suppose $A=\big( \bigsqcup_{s \in S} A_s, c^{A}_i, F^{A}_j, R^{A}_k \big)_{i \in I, j \in J, k \in K}$ is an $\mathcal{L}$-structure that is an \emph{element of} $\NN$. Since every symbol in the language is an element of $\NN$, it follows by $\Sigma_0$-elementarity of $\NN$ in $(V,\in)$ that $\NN \cap A$ contains all the $c^{A}_i$'s and is closed under each $F^{A}_j$ (the fact that the arity of the function symbols is finite is crucial here).  So it makes sense to form the restriction $A \restriction \NN$ on the universe $\NN \cap A$.  In fact this is a fully elementary substructure of $A$ in the language $\L$, since for any $\L$-formula $\varphi$ and any list $a_1,\dots,a_k$ of parameters from $A$, the relation
\[
A \models \varphi(a_1,\dots,a_k)
\]
can be expressed using only bounded quantifiers (namely, bounded quantifiers of the form $\exists z \in A$ and $\forall z \in A$). In particular, for any $\L$-theory $T$, $A \models T$ if and only if $A \restriction \NN \models T$.  Furthermore, if $\pi: A \to B$ is a homomorphism of $\L$-structures with $\pi \in \NN$, then since $\pi \in \NN \prec_0 (V,\in)$, it follows that $\pi(a) \in \NN$ for every $a \in \NN \cap A$.  So $\pi \restriction \NN: A \restriction \NN \to B \restriction \NN$ is an $\L$-homomorphism. To summarise, if $\L$ is finitary and $\L \cup \{ \L \} \subset \NN \prec_0 (V,\in)$, then:
\begin{equation}
\text{for every } \L \text{-theory } T, \ (-) \restriction \NN: \NN \cap \Mod(T) \to \Mod(T) \text{ is a functor.}
\end{equation}

For a small category $\C$ and an $\NN \prec_1 (V,\in)$, by ``$\C \in \NN$" we mean that the ordered tuple
\[
\left( \Obj(\C), \Mor(\C), \circ_\C, \dom_\C, \cod_\C, (id_c)_{c \in \Obj(\C)} \right)
\]
is an element of $\NN$, and it follows by elementarity that 
\[
\C \restriction \NN:= \left(\NN \cap \Obj(\C), \NN \cap \Mor(\C) \right)
\]
is a (not necessarily full) subcategory of $\mathcal{C}$.  
Let $D: \C \to \Mod(T)$ be a functor whose domain is a small category, by ``$D \in \NN$" we will mean that $\C \in \NN$ (in the above sense) and that the sets $\{ (c, D(c)) : c \in \text{Obj} (\C) \}$ and $\{ (f,D(f)) : f \in \text{Mor}(\C)  \}$ are both elements of $\NN$. If that holds, we define $D \restriction \NN$ to be the restriction of $D$ to the category $\C \restriction \NN$, which is a functor from $\C \restriction \NN$ into $\NN \cap \Mod(T)$.

\begin{proposition}[Restriction to subuniverses commutes with colimits]
\thlabel{compute-colimits-correctly}
Let $T$ be a limit theory in some finitary language $\L$ and let $\NN \prec_1 (V, \in)$  with $\L$ and $T$ both elements and subsets of $\NN$. Suppose that $D: \C \to \Mod(T)$ is a small diagram and $D \in \NN$.  Then
\[
(\colim D) \restriction \NN = \colim (D \restriction \NN).
\]
\end{proposition}
\begin{proof}
Let $\Phi(C,D,T,\L)$ denote the relation of $C$ being a colimit of the diagram $D$ in $\Mod(T)$. The usual translation of the relation $\Phi$ via the universal property is $\Pi_3$-expressible in $(V, \in)$. However, by \thref{levy-complexity-colimits-mod-t} we have that $\Phi(C,D,T,\L)$ is $\Sigma_1$-expressible in $(V,\in)$. The proof of \thref{levy-complexity-colimits-mod-t} in the appendix is a routine, but lengthy, analysis of the L\'evy complexity of the explicit colimit construction in $\Mod(T)$, which makes essential use of the assumption that $\L$ is a finitary language.

Modulo \thref{levy-complexity-colimits-mod-t}, the proposition is proved as follows. Let $\L$, $T$, $D$, and $\NN$ be as in the statement of the proposition.  Since $D,\L,T \in \NN \prec_1 (V,\in)$ and $\Phi$ is $\Sigma_1$-expressible, there is a $C \in \NN$ witnessing the $\Sigma_1$-statement $\exists C \ \Phi(C,D,T,\L)$.  Let $\pi: \NN \to \overline{\NN}$ denote the Mostowski collapsing isomorphism \cite[Theorem 6.15]{MR1940513} of $\NN$ with its transitive collapse $\overline{\NN}$, and for each $y \in \NN$, let $\overline{y}:= \pi(y)$.  Since $\L$ and $T$ are both elements and subsets of $\NN$, we may without loss of generality assume they are in the transitive part of $\NN$ and hence $\pi$ fixes both; i.e. $\overline{\L} = \L$ and $\overline{T} = T$.  Then by elementarity of $\pi$, $\overline{\NN}$ satisfies $\Phi\left( \overline{C},\overline{D},\overline{T}=T,\overline{\L} = \L \right)$.  Since $\Phi$ is $\Sigma_1$-expressible, it is upward absolute from the transitive set $\overline{\NN}$ to the universe $V$ by \thref{UpwardAbsolutenessSigma_1}.  So $(V,\in)$ satisfies that $\overline{C}$ is a colimit of $\overline{D}$.  And $\pi$ witnesses that $\overline{D} \simeq D \restriction \NN$ and $\overline{C} \simeq C \restriction \NN$.  So $C \restriction \NN$, i.e., $(\colim D) \restriction \NN$, is a colimit of $D \restriction \NN$.
\end{proof}

If $T$ is a limit theory in $L_{\lambda,\lambda}$ for a possibly uncountable regular $\lambda$, then a similar argument shows that colimits in $\Mod(T)$ commute with restrictions to $\NN$, \textbf{provided that $\NN$ is closed under sequences of length smaller than $\lambda$}.  For $\lambda = \aleph_0$ this requirement holds automatically for any $\NN \prec_0 (V,\in)$, i.e., any such elementary submodel is closed under finite sequences.

\section{Cellular generation in locally finitely presentable categories}
\label{sec:cellular-generation-lfp}

\begin{definition}
\thlabel{quasieffective-square}
Let $\M$ be a class of morphisms in a cocomplete category $\K$. A commutative square is called \emph{$\M$-quasieffective} if two parallel morphisms are in $\M$ and the morphism from the relevant pushout is in $\M$.
\[\begin{tikzcd}
	A && B &&&& B \\
	&&& \Longleftrightarrow & A & P \\
	C && D && C & D
	\arrow["{\in \M}"{description}, from=1-1, to=1-3]
	\arrow[curve={height=-12pt}, from=2-5, to=1-7]
	\arrow[from=2-5, to=2-6]
	\arrow["{\in \M}"{description}, dashed, from=2-6, to=1-7]
	\arrow["\lrcorner"{anchor=center, pos=0.125, rotate=-90}, draw=none, from=2-6, to=3-5]
	\arrow[from=3-1, to=1-1]
	\arrow["{\M\text{-quasieffective}}"{description}, draw=none, from=3-1, to=1-3]
	\arrow["{\in \M}"{description}, from=3-1, to=3-3]
	\arrow[from=3-3, to=1-3]
	\arrow[from=3-5, to=2-5]
	\arrow[from=3-5, to=3-6]
	\arrow[curve={height=12pt}, from=3-6, to=1-7]
	\arrow[from=3-6, to=2-6]
\end{tikzcd}\]
\end{definition}
Even though we will not need it, we note that if $\M$ is composable then the $\M$-quasieffective squares are closed under both horizontal and vertical composition. This follows from the proof of \cite[Theorem 2.7]{liebermanCellularCategoriesStable2023}.
\begin{theorem}
\thlabel{cellular-generation-lfp}
Let $\K$ be a locally finitely presentable category and let $\M$ be a cellularly closed class of morphisms in $\K$. Then the following are equivalent.
\begin{enumerate}[label=(\roman*)]
\item The class $\M$ is cellularly generated.
\item There is a cardinal $\mu$ such that every $f \in \M$, with $\rank_{\K^2}(f) > \mu$, admits a smooth filtration in $\K^2$ that consists of $\M$-quasieffective squares.
\item\label{item_ae_quasi} The class $\M$ is \emph{almost everywhere quasieffective}. That is, for any concrete presentation of $\K$ as $\Mod(T)$, where $T$ is a limit theory in some finitary language $\L$, there exists a regular uncountable $\kappa > |\L|$ and a parameter $p$ such that for every $\NN \prec_1 (V, \in)$, with $p, T, \L, \kappa \in \NN$ and $\NN \cap \kappa$ transitive, the following holds. For every $f: A \to B$ in $\NN \cap \M$ the square below is $\M$-quasieffective.
\[\begin{tikzcd}[sep=small]
	A & B \\
	{A \restriction \NN} & {B \restriction \NN}
	\arrow["f", from=1-1, to=1-2]
	\arrow[hook, from=2-1, to=1-1]
	\arrow["{f \restriction \NN}"', from=2-1, to=2-2]
	\arrow[hook, from=2-2, to=1-2]
\end{tikzcd}\]
\end{enumerate}
\end{theorem}
\begin{proof}
We prove the implications $\text{(ii)} \Rightarrow \text{(i)} \Rightarrow \text{(iii)} \Rightarrow \text{(ii)}$. Each implication will be presented in a separate statement, each with a little bit more detail about the assumptions, cardinals, and parameters involved.
\end{proof}
\begin{remark}
\thlabel{cellular-generation-lfp-remarks}
We make some remarks about \thref{cellular-generation-lfp}.
\begin{enumerate}[label=(\arabic*)]
\item We could not find a direct categorical proof of (i) $\Rightarrow$ (ii). Similarly, \thref{transfinite-compositions-of-cardinal-length} relies on the set-theoretical techniques, but there is a purely categorical version if some extra assumptions are placed on $\M$ (\thref{transfinite-compositions-of-cardinal-length-continuous}). 
\item In (ii), a morphism $f$ admitting a smooth filtration $(f_i)_{i < \lambda}$ in $\K^2$ that consists of $\M$-quasieffective squares means that every morphism in the diagram, as well as every coprojection, is an $\M$-quasieffective square.
\item A ``concrete presentation of $\K$ as $\Mod(T)$'' as in (iii) means that there is an equivalence of categories $F: \Mod(T) \to \K$, and we replace $\M$ with $F^{-1}(\M)$.
\item The terminology ``almost everywhere'' appearing in \ref{item_ae_quasi} comes from set theory, where something happens ``almost everywhere'' if it happens on a set from some fixed filter (typically, the filter generated by the closed unbounded subsets).  Though we do not explicitly use filters here, one \emph{could} rephrase everything in terms of the club filters on $V_\alpha$ for the closed unbounded class of $\alpha$ such that $(V_\alpha, \in) \prec_1 (V,\in)$.  One could also phrase the condition in terms of Shelah's stationary logic, which has a second-order ``almost all'' quantifier, see for example \cite[page 23]{coxMaximumDeconstructibilityModule2022}.
\end{enumerate}
\end{remark}
As a corollary to the proof of \thref{cellular-generation-lfp} we obtain the following improvement of \cite[Lemma 4.16]{makkaiFatSmallObject2014} for locally finitely presentable categories by showing that in a cellularly generated class, every morphism can be expressed as a transfinite composition of cardinal length. See also \thref{transfinite-compositions-of-cardinal-length-continuous} for a version for arbitrary locally presentable categories, but where $\M$ is assumed to be continuous and consist of monomorphisms.
\begin{corollary}
\thlabel{transfinite-compositions-of-cardinal-length}
Let $\M$ be a cellularly generated class of morphisms in a locally finitely presentable category $\K$. Then there is $\mu$ such that for all regular $\kappa \geq \mu$ any morphism $f \in \M$ between $\kappa$-presentable objects can be expressed as a transfinite composition of length exactly $\kappa'$ of morphisms in $\Po(\M_\mu)$, for some (possibly finite) cardinal $\kappa' < \kappa$.

If $\Mod(T)$ is a concrete presentation of $\K$, where $T$ is a limit theory in some finitary language $\L$, then we may take $\mu = (|\L| + \aleph_0)^+$.
\end{corollary}
\begin{proof}
This is the ``moreover'' part of \thref{filtrations-imply-cellular-generation}. The part about the choice of $\mu$ for $\K$ of the form $\Mod(T)$ follows from \thref{lfp-ae-quasieffective-implies-filtrations}.
\end{proof}
\begin{lemma}
\thlabel{build-cellular-generation-from-filtration}
Let $\M$ be a class of morphisms in a cocomplete category $\K$ and let $\kappa$ be a regular cardinal. Suppose that $f \in \M$ is the colimit of a smooth $(< \kappa)$-small chain $(f_i: A_i \to B_i)_{i < \kappa'}$ in $\K^2$, where $\kappa' < \kappa$ is a cardinal, that consists of $\M$-quasieffective squares. Then $f \in \cell(\M_{< \kappa})$. More precisely, $f$ can be expressed as a transfinite composition of length exactly $\kappa'$ of morphisms in $\Po(\M_{< \kappa})$.
\end{lemma}
\begin{proof}
Let $f: A \to B$ and $(f_i: A_i \to B_i)_{i < \kappa'}$ be as in the statement. That is, $\kappa' < \kappa$ and $A_i$ and $B_i$ are $(< \kappa)$-presentable in $\K^2$ for all $i < \kappa'$. We add one more link to the chain: $f_{\kappa'}: A_{\kappa'} \to B_{\kappa'}$ is simply $f: A \to B$, so $(f_i: A_i \to B_i)_{i \leq \kappa'}$ is still smooth.

We build a diagram $(P_{i,j})_{i \leq j \leq \kappa'}$ as pictured below by nested induction.
\[\begin{tikzcd}[column sep=small]
	{A_{\kappa'}} & {P_{1,\kappa'}} & {P_{2,\kappa'}} & {P_{\omega,\kappa'}} & {P_{\omega+1,\kappa'}} & {P_{\kappa',\kappa'} = B_{\kappa'}} \\
	\\
	{A_{\omega+1} = P_{0,\omega+1}} & {P_{1,\omega+1}} & {P_{2,\omega+1}} & {P_{\omega,\omega+1}} & {P_{\omega+1,\omega+1}} & {B_{\omega+1}} \\
	{A_\omega = P_{0,\omega}} & {P_{1,\omega}} & {P_{2,\omega}} & {P_{\omega,\omega}} && {B_\omega} \\
	\\
	{A_2 = P_{0,2}} & {P_{1,2}} & {P_{2,2}} &&& {B_2} \\
	{A_1 = P_{0,1}} & {P_{1,1}} &&&& {B_1} \\
	{A_0 = P_{0,0}} &&&&& {B_0}
	\arrow[from=1-1, to=1-2]
	\arrow[from=1-2, to=1-3]
	\arrow["\lrcorner"{anchor=center, pos=0.125, rotate=-90}, draw=none, from=1-2, to=3-1]
	\arrow[dashed, from=1-3, to=1-4]
	\arrow["\lrcorner"{anchor=center, pos=0.125, rotate=-90}, draw=none, from=1-3, to=3-2]
	\arrow[equal, from=1-4, to=1-5]
	\arrow["\lrcorner"{anchor=center, pos=0.125, rotate=-90}, draw=none, from=1-4, to=3-3]
	\arrow[dashed, from=1-5, to=1-6]
	\arrow["\lrcorner"{anchor=center, pos=0.125, rotate=-90}, draw=none, from=1-5, to=3-4]
	\arrow[dashed, from=3-1, to=1-1]
	\arrow[from=3-1, to=3-2]
	\arrow[dashed, from=3-2, to=1-2]
	\arrow[from=3-2, to=3-3]
	\arrow["\lrcorner"{anchor=center, pos=0.125, rotate=-90}, draw=none, from=3-2, to=4-1]
	\arrow[dashed, from=3-3, to=1-3]
	\arrow[dashed, from=3-3, to=3-4]
	\arrow["\lrcorner"{anchor=center, pos=0.125, rotate=-90}, draw=none, from=3-3, to=4-2]
	\arrow[dashed, from=3-4, to=1-4]
	\arrow[equal, from=3-4, to=3-5]
	\arrow["\lrcorner"{anchor=center, pos=0.125, rotate=-90}, draw=none, from=3-4, to=4-3]
	\arrow[dashed, from=3-5, to=1-5]
	\arrow[from=3-5, to=3-6]
	\arrow["\lrcorner"{anchor=center, pos=0.125, rotate=-90}, draw=none, from=3-5, to=4-4]
	\arrow[dashed, from=3-6, to=1-6]
	\arrow[from=4-1, to=3-1]
	\arrow[from=4-1, to=4-2]
	\arrow[from=4-2, to=3-2]
	\arrow[from=4-2, to=4-3]
	\arrow["\lrcorner"{anchor=center, pos=0.125, rotate=-90}, draw=none, from=4-2, to=6-1]
	\arrow[from=4-3, to=3-3]
	\arrow[dashed, from=4-3, to=4-4]
	\arrow["\lrcorner"{anchor=center, pos=0.125, rotate=-90}, draw=none, from=4-3, to=6-2]
	\arrow[from=4-4, to=3-4]
	\arrow[equal, from=4-4, to=4-6]
	\arrow[from=4-6, to=3-5]
	\arrow[from=4-6, to=3-6]
	\arrow[dashed, from=6-1, to=4-1]
	\arrow[from=6-1, to=6-2]
	\arrow[dashed, from=6-2, to=4-2]
	\arrow[from=6-2, to=6-3]
	\arrow["\lrcorner"{anchor=center, pos=0.125, rotate=-90}, draw=none, from=6-2, to=7-1]
	\arrow[dashed, from=6-3, to=4-3]
	\arrow[from=6-3, to=6-6]
	\arrow["\lrcorner"{anchor=center, pos=0.125, rotate=-90}, draw=none, from=6-3, to=7-2]
	\arrow[dashed, from=6-6, to=4-4]
	\arrow[dashed, from=6-6, to=4-6]
	\arrow[from=7-1, to=6-1]
	\arrow[from=7-1, to=7-2]
	\arrow[from=7-2, to=6-2]
	\arrow[from=7-2, to=7-6]
	\arrow["\lrcorner"{anchor=center, pos=0.125, rotate=-90}, draw=none, from=7-2, to=8-1]
	\arrow[from=7-6, to=6-3]
	\arrow[from=7-6, to=6-6]
	\arrow[from=8-1, to=7-1]
	\arrow[from=8-1, to=8-6]
	\arrow[from=8-6, to=7-2]
	\arrow[from=8-6, to=7-6]
\end{tikzcd}\]
The outer induction is on the rows (the $j$ coordinate in $P_{i,j}$) and the inner induction is on the columns (the $i$ coordinate in $P_{i,j}$).
\begin{itemize}
\item Row $0$ is simply the morphism $f_0: P_{0,0} = A_0 \to B_0$.
\item Having constructed row $j$ we construct row $j+1$ as follows.
\begin{itemize}
\item We set $P_{0,j+1} = A_{j+1}$.
\item For the successor step we assume to have constructed $P_{i,j+1}$, where $i \leq j$. We construct $P_{i+1,j+1}$ based on two cases. If $i = j$, so we are constructing the last link in this row, then we let $P_{i,j+1} \to P_{i+1,j+1} \leftarrow B_j$ be the pushout of $P_{i,j+1} \leftarrow P_{i,j} \to B_j$. Otherwise, if $i < j$, we let $P_{i,j+1} \to P_{i+1,j+1} \leftarrow P_{i+1,j}$ be the pushout of $P_{i,j+1} \leftarrow P_{i,j} \to P_{i+1,j}$.
\item For the limit step we assume to have constructed $P_{i,j+1}$ for all $i < \ell$, where $\ell \leq j$ is a limit. We construct $P_{\ell,j+1}$ as the colimit of the chain $(P_{i,j+1})_{i < \ell}$.
\end{itemize}
The final link $P_{j+1,j+1}$ in the row is a pushout, whose universal property yields a morphism $P_{j+1,j+1} \to B_{j+1}$ so that composing this with the transfinite composition $A_{j+1} \to P_{j+1,j+1}$ yields exactly $f_{j+1}$.
\item Having constructed all rows up to limit $\ell$, we construct row $\ell$ by simply letting $P_{i,\ell}$, for $i < \ell$, be the colimit of the chain $(P_{i,j})_{j < \ell}$. The universal property of this colimit then yields the morphisms $P_{i,\ell} \to P_{i',\ell}$ for all $i < i' \leq \ell$ (here we make use of the fact that $P_{0, \ell} = A_\ell = \colim_{j < \ell} A_j = \colim_{j < \ell} P_{0,j}$). Finally, we set $P_{\ell,\ell} = B_\ell$.
\end{itemize}
The following facts about the construction can be proved by induction.
\begin{enumerate}[label=(F\arabic*)]
\item For any $\ell \leq j \leq \kappa'$ with $\ell$ limit, $P_{\ell,j}$ is the colimit of the chain $(P_{i,j})_{i < \ell}$. That is, the chain $(P_{i,j})_{i \leq j}$ is continuous. Indeed, if $j$ is a successor then this holds by construction. If $j$ is a limit then for $\ell < j$ this follows by induction: the statement holds for the rows below $j$ and so $P_{\ell,j} = \colim_{j' < j} P_{\ell,j'} = \colim_{j' < j} \colim_{i < \ell} P_{i,j'} = \colim_{i < \ell} \colim_{j' < j} P_{i,j'} = \colim_{i < \ell} P_{i, j}$. For $j = \ell$ we have that $P_{\ell,\ell} = B_\ell = \colim_{j < \ell} B_j = \colim_{j < \ell} P_{j,j} = \colim_{i \leq j < \ell} P_{i,j} = \colim_{i < \ell} P_{i, \ell}$.
\item Any rectangle of $P_{i,j}$'s is a pushout. More precisely, for any $i' \leq i \leq j' \leq j \leq \kappa'$ we have that the square below is a pushout square.
\[\begin{tikzcd}[sep=small]
	{P_{i',j}} & {P_{i,j}} \\
	{P_{i',j'}} & {P_{i,j'}}
	\arrow[from=1-1, to=1-2]
	\arrow["\lrcorner"{anchor=center, pos=0.125, rotate=-90}, draw=none, from=1-2, to=2-1]
	\arrow[from=2-1, to=1-1]
	\arrow[from=2-1, to=2-2]
	\arrow[from=2-2, to=1-2]
\end{tikzcd}\]
This is seen by induction on the ordinals $\alpha$ and $\beta$ such that $i = i' + \alpha$ and $j = j' + \beta$. Successor steps follow from the pasting law for pushouts and limit stages follow because colimits of chains (see the construction and (F1)) commute with pushouts.
\item For all $j' < j \leq \kappa'$ we have that $A_{j} \to P_{j'+1,j} \leftarrow B_{j'}$ is the pushout of $A_j \leftarrow A_{j'} \to B_{j'}$. This follows from composition of the pushouts below.
\[\begin{tikzcd}[sep=small]
	{A_j = P_{0,j}} & {P_{j',j}} & {P_{j'+1,j}} \\
	{A_{j'+1} =P_{0,j'+1}} & {P_{j',j'+1}} & {P_{j'+1,j'+1}} \\
	{A_{j'} = P_{0,j'}} & {P_{j',j'}} & {B_{j'}}
	\arrow[from=1-1, to=1-2]
	\arrow[from=1-2, to=1-3]
	\arrow["\lrcorner"{anchor=center, pos=0.125, rotate=-90}, draw=none, from=1-2, to=2-1]
	\arrow["\lrcorner"{anchor=center, pos=0.125, rotate=-90}, draw=none, from=1-3, to=2-2]
	\arrow[from=2-1, to=1-1]
	\arrow[from=2-1, to=2-2]
	\arrow[from=2-2, to=1-2]
	\arrow[from=2-2, to=2-3]
	\arrow["\lrcorner"{anchor=center, pos=0.125, rotate=-90}, draw=none, from=2-2, to=3-1]
	\arrow[from=2-3, to=1-3]
	\arrow["\lrcorner"{anchor=center, pos=0.125, rotate=-90}, draw=none, from=2-3, to=3-2]
	\arrow[from=3-1, to=2-1]
	\arrow[from=3-1, to=3-2]
	\arrow[from=3-2, to=2-2]
	\arrow[from=3-2, to=3-3]
	\arrow[from=3-3, to=2-3]
\end{tikzcd}\]
Here all squares are pushouts by (F2), except for the bottom right square, which is a pushout by construction.
\item Every object, except those in row $\kappa'$, is $(< \kappa)$-presentable. For the $A_j$'s and $B_j$'s this is by assumption. For the $P_{i,j}$'s this is because they are always constructed as a colimit of a $\kappa'$-small diagram of objects that are, by the induction hypothesis, $(< \kappa)$-presentable (using that $\kappa$ is regular). Here it is important that $\kappa'$ is a cardinal.
\end{enumerate}
We have now expressed $f: A \to B$ as the transfinite composition of the top row in the diagram. Indeed, this top row forms a continuous chain by (F1). It remains to show that each link of the top row is in $\cell(\M_{< \kappa})$. Fix $i < \kappa'$. The link $P_{i,\kappa'} \to P_{i+1,\kappa'}$ is a pushout of $P_{i,i} \to B_i$, by (F2) and composition of pushouts. If $i$ is a limit then $P_{i,i} \to B_i$ is the identity by construction, so $P_{i,\kappa'} \to P_{i+1,\kappa'}$ is an isomorphism and is thus in $\Po(\M_{< \kappa})$. If $i$ is a successor, say $i = j+1$, then we recall that the square
\[\begin{tikzcd}[sep=small]
	{A_i} & {B_i} \\
	{A_j} & {B_j}
	\arrow[from=1-1, to=1-2]
	\arrow[from=2-1, to=1-1]
	\arrow[from=2-1, to=2-2]
	\arrow[from=2-2, to=1-2]
\end{tikzcd}\]
is $\M$-quasieffective. As $P_{i,i}$ is the pushout of $A_i \leftarrow A_j \to B_j$ by (F3), we have that $P_{i,i} \to B_i$ is in $\M$. Finally, $P_{i,i}$ and $B_i$ are $(< \kappa)$-presentable by (F4), so $P_{i,i} \to B_i$ is in fact in $\M_{<\kappa}$. We conclude that $P_{i,\kappa'} \to P_{i+1,\kappa'}$ is in $\Po(\M_{< \kappa})$, as required.
\end{proof}
\begin{proposition}[{(ii) $\Rightarrow$ (i) in \thref{cellular-generation-lfp}}]
\thlabel{filtrations-imply-cellular-generation}
Let $\M$ be a cellularly closed class of morphisms in a locally $\lambda$-presentable category $\K$. Suppose that $\mu \geq \lambda$ is such that every $f \in \M$, with $\rank_{\K^2}(f) > \mu$, admits a smooth filtration in $\K^2$ that consists of $\M$-quasieffective squares. Then $\M = \cell(\M_\mu)$, so $\M$ is cellularly generated.

Moreover, for any $\kappa \geq \mu$ and any $f \in \M_\kappa$ there is a (possibly finite) cardinal $\kappa' < \kappa$ such that $f$ is a transfinite composition of length exactly $\kappa'$ of morphisms in $\Po(\M_\mu)$.
\end{proposition}
\begin{proof}
We prove the statement, including the ``moreover'' part, for $\M_\kappa$ by induction on the cardinal $\kappa$. The statement is trivial for $\kappa = \mu$. Furthermore, for a limit cardinal $\kappa$ we have $\M_\kappa = \bigcup_{\kappa' < \kappa} \M_{\kappa'}$ by \thref{successor-presentability-rank} (using that $\mu \geq \lambda$), so limit stages trivially follow from the induction hypothesis.

Now assume we have proved the statement for $\M_\kappa$. Let $f \in \M_{\kappa^+} \setminus \M_\kappa$. By assumption $f$ admits a smooth filtration $(f_i)_{i < \delta}$ in $\K^2$ that consists of $\M$-quasieffective squares. So $\delta < \kappa^+$, hence $\cf(\delta) = \cf(\kappa)$ and by taking a cofinal subchain we may assume $\delta = \cf(\kappa)$. By \thref{build-cellular-generation-from-filtration} we then have that $f$ can be expressed as a transfinite composition of some chain $(g_{ij})_{i \leq j < \delta}$, where $g_{i,i+1}$ is a pushout of some $g'_{i,i+1} \in \M_{\kappa}$ for each $i < \delta$. By the induction hypothesis we have that $g'_{i,i+1}$ is a transfinite composition of cardinal length $\kappa_i < \kappa$ of morphisms in $\Po(\M_\mu)$. Since pushouts commute with transfinite compositions, we have that $g_{i,i+1}$ is a transfinite composition of length $\kappa_i$ of morphisms in $\Po(\M_\mu)$. We conclude that $f$ is a transfinite composition whose length is the ordinal sum $\sum_{i < \delta} \kappa_i$, which is at most $\kappa$ because $\delta = \cf(\kappa)$. By inserting identity morphisms if necessary we may assume this composition is of length $\kappa$, as required.
\end{proof}
\begin{proposition}[{(iii) $\Rightarrow$ (ii) in \thref{cellular-generation-lfp}}]
\thlabel{lfp-ae-quasieffective-implies-filtrations}
Let $T$ be a limit theory in some finitary language $\L$ and let $\M$ be a class of morphisms in $\Mod(T)$. Suppose that there is a regular uncountable $\kappa > |\L|$ and a parameter $p$ such that for every $\NN \prec_1 (V, \in)$, with $p,T,\L,\kappa \in \NN$ and $\NN \cap \kappa$ transitive, the following holds. For every $f: A \to B$ in $\NN \cap \M$ the square below is $\M$-quasieffective.
\[\begin{tikzcd}[sep=small]
	A & B \\
	{A \restriction \NN} & {B \restriction \NN}
	\arrow["f", from=1-1, to=1-2]
	\arrow[hook, from=2-1, to=1-1]
	\arrow["{f \restriction \NN}"', from=2-1, to=2-2]
	\arrow[hook, from=2-2, to=1-2]
\end{tikzcd}\]
Then for all $f \in \M$ with $\rank_{\K^2}(f) > (|\L| + \aleph_0)^+$ there is a smooth filtration in $\K^2$ that consists of $\M$-quasieffective squares.
\end{proposition}
\begin{proof}
Let $f \in \M$ with $\rank_{\K^2}(f) > (|\L| + \aleph_0)^+$, so $\rank_{\K^2}(f) = \mu^+$ for some cardinal $\mu > |\L| + \aleph_0$ by \thref{successor-presentability-rank}. By \thref{Class_LS_Theorem} we can inductively construct a chain $(\NN_i)_{i \leq \cf(\mu)}$ such that for each $i \leq \cf(\mu)$:
\begin{enumerate}[label=(\arabic*)]
\item $p,\L,\kappa,f \in \NN_i \prec_1 (V, \in)$ and $\NN_i \cap \kappa$ is transitive;
\item for all $j < i$, all parts of the span $A \supseteq A \restriction \NN_j \xrightarrow{f \restriction \NN_j} B \restriction \NN_j$ are elements of $\NN_i$, as well as the induced map from its pushout to $B$;
\item if $i < \cf(\mu)$ then $|\NN_i| < \mu$;
\item for all $j \leq i$ we have $\NN_j \subseteq \NN_i$;
\item if $i$ is a limit then $\NN_i = \bigcup_{j < i} \NN_j$;
\item $A, B \subseteq \NN_{\cf(\mu)}$.
\end{enumerate}
The inclusions $A \restriction \NN_i \subseteq A \restriction \NN_j$ and $B \restriction \NN_i \subseteq B \restriction \NN_j$ for $i \leq j \leq \cf(\mu)$ make $(f \restriction \NN_i)_{i \leq \cf(\mu)}$ into a chain in $\Mod(T)^2$. By (5) and \thref{limit-theory-standard-facts}(ii) this is a smooth chain. Furthermore, $f \restriction \NN_i$ is $(< \mu^+)$-presentable for all $i < \cf(\mu)$, by (3) and \thref{limit-theory-standard-facts}(iii). We thus see that $(f \restriction \NN_i)_{i < \cf(\mu)}$ is a filtration of $f \restriction \NN_{\cf(\mu)}$, which is $f$ by (6), in $\Mod(T)^2$. It thus remains to show that the filtration consists of $\M$-quasieffective squares.

Let $j \leq i \leq \cf(\mu)$, we have to show that all squares consisting of solid arrows in the diagram below (including the outer square) are $\M$-quasieffective.
\[\begin{tikzcd}[column sep=small]
	A && B \\
	& Q \\
	{A \restriction \NN_i} && {B \restriction \NN_i} \\
	& P \\
	{A \restriction \NN_j} && {B \restriction \NN_j}
	\arrow["f", from=1-1, to=1-3]
	\arrow[curve={height=12pt}, dashed, from=1-1, to=2-2]
	\arrow["u", dashed, from=2-2, to=1-3]
	\arrow["\lrcorner"{anchor=center, pos=0.125, rotate=-90}, draw=none, from=2-2, to=3-1]
	\arrow[hook, from=3-1, to=1-1]
	\arrow["{f \restriction \NN_i}"{pos=0.3}, from=3-1, to=3-3]
	\arrow[curve={height=12pt}, dashed, from=3-1, to=4-2]
	\arrow[hook, from=3-3, to=1-3]
	\arrow["v", dashed, from=4-2, to=3-3]
	\arrow["\lrcorner"{anchor=center, pos=0.125, rotate=-90}, draw=none, from=4-2, to=5-1]
	\arrow[hook, from=5-1, to=3-1]
	\arrow["{f \restriction \NN_j}"', from=5-1, to=5-3]
	\arrow[curve={height=-18pt}, dashed, from=5-3, to=2-2]
	\arrow[hook, from=5-3, to=3-3]
	\arrow[curve={height=-12pt}, dashed, from=5-3, to=4-2]
\end{tikzcd}\]
The outer square and top square are $\M$-quasieffective by assumption. It remains to show that $v \in \M$. By (2) we can use that restriction to $\NN_i$ commutes with colimits (\thref{compute-colimits-correctly}) to see that $P = Q \restriction \NN_i$. So $v = u \restriction \NN_i$. By (2) again, we have $u \in \NN_i$ and so we can apply the main assumption to $\NN_i$ and conclude that $v \in \M$.
\end{proof}
\begin{proposition}[{(i) $\Rightarrow$ (iii) in \thref{cellular-generation-lfp}}]
\thlabel{lfp-cellular-generation-implies-ae-quasieffective}
Let $T$ be a limit theory in some language $\L$. Let $\M_0$ be a set of morphisms in $\Mod(T)$ and write $\M := \cell(\M_0)$. Let $\kappa > |\M_0| + |\L| + \aleph_0$ be such that every $f \in \M_0$ is $\kappa$-presentable in $\Mod(T)^2$. Then for every $\NN \prec_1 (V, \in)$ with $\M_0,T,\L,\kappa \in \NN$ and $\NN \cap \kappa$ transitive the following holds. For every $f: A \to B$ in $\NN \cap \M$ the square below is $\M$-quasieffective.
\[\begin{tikzcd}[sep=small]
	A & B \\
	{A \restriction \NN} & {B \restriction \NN}
	\arrow["f", from=1-1, to=1-2]
	\arrow[hook, from=2-1, to=1-1]
	\arrow["{f \restriction \NN}"', from=2-1, to=2-2]
	\arrow[hook, from=2-2, to=1-2]
\end{tikzcd}\]
\end{proposition}
\begin{proof}
Let $f: A \to B$ be in $\NN \cap \M$. By assumption, we can write $f$ as the transfinite composition of some smooth chain
\[
(f_{i,j}: A_i \to A_j)_{i \leq j \leq \delta},
\]
with $A_0 = A$ and $A_\delta = B$, such that $f_{i,i+1}$ is a pushout of some $C_i \to D_i$ in $\M_0$ for all $i < \delta$. Note that by our choice of $\kappa$ we have $\M_0 \subseteq \NN$, as well as $C_i, D_i \subseteq \NN$ for all $i < \delta$ (see \thref{kappa-sized-is-subset-of-subuniverse}).

Formally this means that we have a functor $F: \delta+1 \to \Mod(T)$, viewing $\delta+1$ as a poset and thus as a category, given by $F(i) = A_i$ and $F(i \leq j) = f_{i,j}$. By elementarity of $\NN$ we can assume that $F \in \NN$, and hence $\delta \in \NN$. Furthermore, we can assume that for all $i \leq \delta$ with $i \in \NN$ all components of the pushout diagram below are elements of $\NN$.
\begin{equation}
\label{lfp-cellular-generation-implies-ae-quasieffective:pushout-ci-di}
\begin{tikzcd}
	{A_i} & {A_{i+1}} \\
	{C_i} & {D_{i+1}}
	\arrow["{f_{i,i+1}}", from=1-1, to=1-2]
	\arrow["\lrcorner"{anchor=center, pos=0.125, rotate=-90}, draw=none, from=1-2, to=2-1]
	\arrow[from=2-1, to=1-1]
	\arrow[from=2-1, to=2-2]
	\arrow[from=2-2, to=1-2]
\end{tikzcd}
\end{equation}
Consider the chain
\[
F \restriction \NN = (f_{i,j} \restriction \NN: A_i \restriction \NN \to A_j \restriction \NN)_{\substack{i \leq j \leq \delta\\ i,j \in \NN}}.
\]
Firstly, this is well-defined, because for $i,j \in \NN$ with $i \leq j \leq \delta$ we have $f_{i,j} \in \NN$. Secondly, this is a smooth chain. Indeed, let $\ell \leq \delta$ with $\ell \in \NN$ be a limit ordinal. Then $(F \restriction \NN) \restriction \ell = (F \restriction \ell) \restriction \NN$. Furthermore, the restriction $F \restriction \ell$ of the original chain to $\ell$ is an element of $\NN$. So by \thref{compute-colimits-correctly} we have
\[
\colim ((F \restriction \NN) \restriction \ell) =
\colim ((F \restriction \ell) \restriction \NN) =
(\colim F \restriction \ell) \restriction \NN =
A_\ell \restriction \NN,
\]
as required.

For any ordinal $i$ we have $i \in \NN$ if and only if $i+1 \in \NN$. So the links in the chain $F \restriction \NN$ are all of the form $f_{i,i+1} \restriction \NN$, where $i < \delta$ and $i \in \NN$. For such $i$ we can restrict the pushout diagram in (\ref{lfp-cellular-generation-implies-ae-quasieffective:pushout-ci-di}) to the diagram below.
\[\begin{tikzcd}
	{A_i \restriction \NN} & {A_{i+1} \restriction \NN} \\
	{C_i \restriction \NN = C_i} & {D_i \restriction \NN = D_i}
	\arrow["{f_{i,i+1} \restriction \NN}", from=1-1, to=1-2]
	\arrow["\lrcorner"{anchor=center, pos=0.125, rotate=-90}, draw=none, from=1-2, to=2-1]
	\arrow[from=2-1, to=1-1]
	\arrow[from=2-1, to=2-2]
	\arrow[from=2-2, to=1-2]
\end{tikzcd}\]
This diagram is again a pushout because the original diagram is an element of $\NN$ and restriction to $\NN$ commutes with colimits (\thref{compute-colimits-correctly}). We thus see that every link in $F \restriction \NN$ is a pushout of some morphism in $\M_0$. As $0,\delta \in \NN$ we have by smoothness of $F \restriction \NN$ that $f_{0,\delta} \restriction \NN$ is the transfinite composition of $F \restriction \NN$. We thus conclude that $f \restriction \NN = f_{0,\delta} \in \M$.

It remains to show that the induced morphism from the pushout of $A \hookleftarrow A \restriction \NN \xrightarrow{f \restriction \NN} B \restriction \NN$ into $B$ is in $\M$. We will do this by showing that it is a transfinite composition of morphisms in $\M$.

Let $\alpha \leq \delta$ and set
\[
R_\alpha := \colim_{\substack{i \leq \alpha\\ i \in \NN}} A_i \restriction \NN.
\]
Then for any $\beta \geq \alpha$ there is a cocone into $A_\beta$ given by
\[
\left( A_i \restriction \NN \hookrightarrow A_i \xrightarrow{f_{i,\beta}} A_\beta \right)_{\substack{i \leq \alpha\\ i \in \NN}},
\]
and if $\beta \in \NN$ there is also cocone into $A_\beta \restriction \NN$ given by
\[
\left( A_i \restriction \NN \xrightarrow{f_{i,\beta} \restriction \NN} A_\beta \restriction \NN \right)_{\substack{i \leq \alpha\\ i \in \NN}}.
\]
These cocones induce morphisms $R_\alpha \to A_\beta$ and $R_\alpha \to A_\beta \restriction \NN$. Using these morphisms, we define $P_\alpha$ as the pushout in the diagram below.
\[\begin{tikzcd}[sep=small]
	{A_\alpha} & {P_\alpha} \\
	{R_\alpha} & {A_\delta \restriction \NN}
	\arrow[from=1-1, to=1-2]
	\arrow["\lrcorner"{anchor=center, pos=0.125, rotate=-90}, draw=none, from=1-2, to=2-1]
	\arrow[from=2-1, to=1-1]
	\arrow[from=2-1, to=2-2]
	\arrow[from=2-2, to=1-2]
\end{tikzcd}\]
Then for $\beta \geq \alpha$ there is an obvious induced morphism $R_\alpha \to R_\beta$ and this induces a morphism $P_\alpha \to P_\beta$, as pictured in the diagram below.
\[\begin{tikzcd}[sep=small]
	& {A_\beta} && {P_\beta} \\
	{A_\alpha} && {P_\alpha} \\
	& {R_\beta} \\
	{R_\alpha} && {A_\delta \restriction \NN}
	\arrow[from=1-2, to=1-4]
	\arrow[from=2-1, to=1-2]
	\arrow[from=2-1, to=2-3]
	\arrow[from=2-3, to=1-4]
	\arrow["\lrcorner"{anchor=center, pos=0.125, rotate=-90}, draw=none, from=2-3, to=4-1]
	\arrow[from=3-2, to=1-2]
	\arrow[from=3-2, to=4-3]
	\arrow[from=4-1, to=2-1]
	\arrow[from=4-1, to=3-2]
	\arrow[from=4-1, to=4-3]
	\arrow[from=4-3, to=1-4]
	\arrow[from=4-3, to=2-3]
\end{tikzcd}\]
The chain $(R_\alpha)_{\alpha \leq \delta}$ is smooth. Indeed, for a limit ordinal $\ell$ we have by definition that $\colim_{\alpha < \ell} R_\alpha = \colim_{\substack{i < \ell\\ i \in \NN}} A_i \restriction \NN$. If $\ell \not \in \NN$ then the right-hand side is $R_\ell$ by definition. If $\ell \in \NN$ then $R_\ell = A_\ell \restriction \NN$ and the claim follows from smoothness of $F \restriction \NN$.

As $(A_i)_{i \leq \delta}$ is also a smooth chain, and because colimits commute, it follows that $(P_i)_{i \leq \delta}$ is a smooth chain. One also easily computes $P_\delta = A_\delta = B$ and that $P_0$ is the pushout of $A \hookleftarrow A \restriction \NN \xrightarrow{f \restriction \NN} B \restriction \NN$. So the induced morphism $P_0 \to B$ is the transfinite composition of the chain $(P_i)_{i \leq \delta}$. It thus suffices to show that $P_\alpha \to P_{\alpha+1}$ is in $\M$ for all $\alpha < \delta$. We distinguish two cases: either $\alpha \not \in \NN$ or $\alpha \in \NN$.

\textbf{The case where $\alpha \not \in \NN$.} Then $\alpha + 1 \not \in \NN$. Hence $R_\alpha = R_{\alpha+1}$ and $R_{\alpha+1} \to A_{\alpha+1}$ factors as
\[
R_{\alpha+1} = R_\alpha \to A_\alpha \xrightarrow{f_{\alpha,\alpha+1}} A_{\alpha+1}.
\]
We thus have a diagram like below.
\[\begin{tikzcd}
	{A_{\alpha+1}} & {P_{\alpha+1}} \\
	{A_\alpha} & {P_\alpha} \\
	{R_\alpha = R_{\alpha+1}} & {A_\delta \restriction \NN}
	\arrow[from=1-1, to=1-2]
	\arrow["\lrcorner"{anchor=center, pos=0.125, rotate=-90}, draw=none, from=1-2, to=2-1]
	\arrow[from=2-1, to=1-1]
	\arrow[from=2-1, to=2-2]
	\arrow[from=2-2, to=1-2]
	\arrow["\lrcorner"{anchor=center, pos=0.125, rotate=-90}, draw=none, from=2-2, to=3-1]
	\arrow[from=3-1, to=2-1]
	\arrow[from=3-1, to=3-2]
	\arrow[from=3-2, to=2-2]
\end{tikzcd}\]
By definition the bottom square and outer rectangle are pushouts. So the top square is a pushout. We thus see that $P_\alpha \to P_{\alpha+1}$ is a pushout of $f_{\alpha,\alpha+1}$, which is in $\M$. So $P_\alpha \to P_{\alpha+1}$ is in $\M$.

\textbf{The case where $\alpha \in \NN$.} Then $\alpha + 1 \in \NN$. By assumption, there is $C_\alpha \to D_\alpha$ such that $f_{\alpha,\alpha+1}$ is a pushout as pictured in (\ref{lfp-cellular-generation-implies-ae-quasieffective:pushout-ci-di}), with $\alpha$ in the role of $i$. As noted before, $C_\alpha$ and $D_\alpha$ are both elements and subsets of $\NN$. By definition $R_\alpha = A_\alpha \restriction \NN$ and $R_{\alpha+1} = A_{\alpha+1} \restriction \NN$. So because restriction to $\NN$ commutes with colimits (\thref{compute-colimits-correctly}) the square below is a pushout square.
\begin{equation}
\label{lfp-cellular-generation-implies-ae-quasieffective:pushout-ca-da-restricted}
\begin{tikzcd}[row sep=small]
	{R_\alpha = A_\alpha \restriction \NN} & {R_{\alpha+1} = A_{\alpha+1} \restriction \NN} \\
	{C_\alpha} & {D_\alpha}
	\arrow["{f_{\alpha,\alpha+1} \restriction \NN}", from=1-1, to=1-2]
	\arrow["\lrcorner"{anchor=center, pos=0.125, rotate=-90}, draw=none, from=1-2, to=2-1]
	\arrow[from=2-1, to=1-1]
	\arrow[from=2-1, to=2-2]
	\arrow[from=2-2, to=1-2]
\end{tikzcd}
\end{equation}
We now collect everything in the commutative diagram consisting of the solid arrows below.
\[\begin{tikzcd}[sep=small]
	&&& {A_{\alpha+1}} && {P_{\alpha+1}} \\
	\\
	& {A_\alpha} &&& {P_\alpha} \\
	&& {D_\alpha} \\
	&&& {R_{\alpha+1} = A_{\alpha+1} \restriction \NN} \\
	{C_\alpha} \\
	& {R_\alpha = A_\alpha \restriction \NN} &&& {A_\delta \restriction \NN}
	\arrow[from=1-4, to=1-6]
	\arrow[dashed, from=1-4, to=3-5]
	\arrow[from=3-2, to=1-4]
	\arrow[from=3-2, to=3-5]
	\arrow[from=3-5, to=1-6]
	\arrow[from=4-3, to=1-4]
	\arrow[from=4-3, to=5-4]
	\arrow[from=5-4, to=1-4]
	\arrow[from=5-4, to=7-5]
	\arrow[from=6-1, to=3-2]
	\arrow[from=6-1, to=4-3]
	\arrow[from=6-1, to=7-2]
	\arrow[from=7-2, to=3-2]
	\arrow[from=7-2, to=5-4]
	\arrow[from=7-2, to=7-5]
	\arrow[from=7-5, to=1-6]
	\arrow[from=7-5, to=3-5]
\end{tikzcd}\]
There are two paths $C_\alpha \to P_\alpha$, namely $C_\alpha \to A_\alpha \to P_\alpha$ and $C_\alpha \to D_\alpha \to R_{\alpha+1} \to A_\delta \restriction \NN \to P_\alpha$. By commutativity of the diagram the compositions of these paths are the same. So we find the dashed arrow because $A_{\alpha+1}$ is a pushout, see (\ref{lfp-cellular-generation-implies-ae-quasieffective:pushout-ci-di}). Furthermore:
\begin{align*}
A_\alpha \to A_{\alpha+1} \dashrightarrow P_\alpha &= A_\alpha \to P_\alpha &\text{(by definition)}, \\
A_{\alpha+1} \dashrightarrow P_\alpha \to P_{\alpha+1} &= A_{\alpha+1} \to P_\alpha &\text{(as $A_{\alpha+1}$ is a pushout (\ref{lfp-cellular-generation-implies-ae-quasieffective:pushout-ci-di}))},\\
R_{\alpha+1} \to A_{\alpha+1} \dashrightarrow P_\alpha &= R_{\alpha+1} \to A_\delta \restriction \NN \to P_\alpha &\text{(as $R_{\alpha+1}$ is a pushout (\ref{lfp-cellular-generation-implies-ae-quasieffective:pushout-ca-da-restricted}))}.
\end{align*}
So the whole diagram, this time including the dashed arrow, commutes. We thus get an induced arrow $P_{\alpha+1} \to P_\alpha$, which is easily checked to be the inverse of $P_\alpha \to P_{\alpha+1}$. We conclude that $P_\alpha \to P_{\alpha+1}$ is an isomorphism and is thus in $\M$.
\end{proof}

\section{Cellular generation in any locally presentable category}
\label{sec_AnyLocallyPresent}

To prove a version of \thref{cellular-generation-lfp} for arbitrary locally presentable categories, one might want to use that these can be presented as $\Mod(T)$ for some limit theory $T$. However, to do this, we have to allow the language $\L$ to be \emph{infinitary}. This is problematic, because to make sense of $A \restriction \NN$ for some $\L$-structure $A$ and some $\NN \prec_n (V, \in)$ we implicitly use that $\L$ is finitary (no matter how big $n$ is). Identifying $A$ with its underlying set we have $(A \cap \NN)^k = A^k \cap \NN$ for any positive integer $k$, which is essential for restricting $k$-ary function symbols to a structure with underlying set $A \cap \NN$. For infinitary $\L$ this no longer works. For example, it can happen that $A^\omega \cap \NN \subsetneq (A \cap \NN)^\omega$, so we can generally not restrict $\omega$-ary function symbols to a structure with underlying set $A \cap \NN$.

We circumvent this issue by using a different kind of presentation of locally presentable categories, namely as full reflective subcategories of presheaf categories. We thank Marc Olschok for suggesting this approach to the first-named author at the CT2025 conference.
\begin{definition}
\thlabel{presheaf-presentation}
We call $\A$ a \emph{$\lambda$-presheaf presentation} of $\K$ if $\K$ is locally $\lambda$-presentable and $\A$ is a small full subcategory of $\lambda$-presentable objects in $\K$, containing at least one representative in each isomorphism class.
\end{definition}
\begin{fact}[{\cite[Theorem 1.46]{adamekLocallyPresentableAccessible1994}}]
\thlabel{presheaf-presentations-exist}
Let $\A$ be a $\lambda$-presheaf presentation of $\K$. Then the functor
\begin{align*}
E: \K &\to \Set^{\A^\op}, \\
K & \mapsto \Hom(-, K) \restriction_{\A^\op},
\end{align*}
is full and faithful and has a left adjoint $\RR: \Set^{\A^\op} \to \K$, called the \emph{reflection functor}.
\end{fact}
If $\A$ is a $\lambda$-presheaf presentation of $\K$ then we can define $\K'$ to be the essential image of $E$ (i.e., the full subcategory of $\Set^{\A^\op}$ on those objects that are isomorphic to one in the image of $E$). Since $E$ is full and faithful, it induces an equivalence of categories $\K \cong \K'$. We may thus view $\K$ as a full subcategory of $\Set^{\A^\op}$, which really means that everything is translated through $E$ (cf.\ \thref{cellular-generation-lfp-remarks}(3)). For notational convenience we make the following convention.
\begin{convention}
\thlabel{presheaf-presentation-subcategory}
Given a a $\lambda$-presheaf presentation $\A$ of $\K$, we will view $\K$ as a full subcategory of $\Set^{\A^\op}$, closed under isomorphisms.
\end{convention}
Suppose $\L$ is a finitary sorted language and $T$ is a first-order $\L$-theory (not necessarily a limit theory). Section \ref{sec_SetTheorySetup} defined the restriction functor 
\[
(-) \restriction \NN: \NN \cap \Mod(T) \to \Mod(T)
\]
for any $\NN$ such that $\L \cup \{ \L \} \subset \NN$ and $\NN \prec_0 (V,\in)$.  A presheaf category $\Set^{\A^\op}$ can be viewed as a special instance of a category of the form $\Mod(T)$, where $\L$ has a sort for every object in $\A^\op$, a unary function symbol for every morphism in $\A^\op$, and axioms echoing the composition in $\A^\op$. We refer to \cite[Section 2]{coxCofibrantGenerationPure2025} for details. In particular, if we have a $\lambda$-presheaf presentation $\A$ of $\K$, and $\A \cup \{ \A \} \subset \NN \prec_1 (V,\in)$, then we can already make sense of restrictions to $\NN$ of presheaves; i.e., we have a \emph{presheaf restriction functor}
\begin{equation}
\label{eq_PresheafRestrictionFunctor}
(-) \restriction \NN: \NN \cap \Set^{\A^\op} \to \Set^{\A^\op}
\end{equation}
Here $\A \subset \NN$ means that $\Obj(\A) \cup \Mor(\A) \subset \NN$.
\begin{definition}
\thlabel{restriction-through-presheaf-presentation}
Let $\A$ be a $\lambda$-presheaf presentation of $\K$ with reflection functor $\RR$, and assume
\[
\A \cup \{ \A \} \subset \NN \prec_1 (V, \in).
\]
For $K \in \NN \cap \Obj(\K)$ and $f \in \NN \cap \Mor(\K)$, we define 
\[
K \restriction^{\K} \NN := \RR(K \restriction \NN) \text{ and } f \restriction^{\K} \NN := \RR(f \restriction \NN),
\]
where $(-) \restriction \NN$ is the presheaf restriction functor \eqref{eq_PresheafRestrictionFunctor}.
\end{definition}
\begin{fact}
\thlabel{big-enough-n-for-subuniverses}
There is a big enough $n < \omega$ such that for any $\lambda$-presheaf presentation of $\K$, with reflection functor $\RR$, we have that:
\begin{enumerate}[label=(\roman*)]
\item the category $\K$ is $\Sigma_n$-definable from parameters $\A$ and $\lambda$;
\item for any diagram $D: \C \to \K$, the relation $C \cong \colim_\K D$ is $\Sigma_n$-expressible from parameters $D$, $\A$ and $\lambda$; 
\item the functor $\RR$ is $\Sigma_n$-expressible from parameters $\A$ and $\lambda$.
\end{enumerate}
More precisely, (iii) means that the relation $\RR(X) \cong Y$ is $\Sigma_n$-expressible, for both objects and morphisms.
\end{fact}
\begin{proof}
By \cite[Theorem 2.26]{adamekLocallyPresentableAccessible1994} the category $\K$ consists of precisely those objects which are $\lambda$-directed colimits of objects in $\A$, which is first-order expressible (and its L\'evy-complexity does not depend on $\K$, $\A$ or $\lambda$). Given the definability of $\K$, item (ii) quickly follows by simply expressing the universal property of the colimit. Finally, $\RR(X)$ can be computed as the colimit of the category of elements of $X$ (see \cite[Proposition 1.27]{adamekLocallyPresentableAccessible1994}), the diagram of which is first-order definable from $X$, and so $\RR$ is first-order definable.
\end{proof}
\begin{remark}
\thlabel{n-is-four-works}
Working out the proof sketch in \thref{big-enough-n-for-subuniverses} carefully, one sees that $n = 4$ works. In fact, for item (i), $n = 2$ already works and items (ii) and (iii) can be reduced to $\Pi_3$-expressibility. We did not think this detailed analysis was worth writing out, hence the above approach.
\end{remark}
\begin{proposition}
\thlabel{subuniverse-closed-under-reflection}
Let $\A$ be a $\lambda$-presheaf presentation of $\K$, and let $n$ be as in \thref{big-enough-n-for-subuniverses}. If $\A \cup \{\A, \lambda\} \subset \NN \prec_n (V, \in)$ then $\NN$ is closed under the reflection functor $\RR: \Set^{\A^\op} \to \K$. That is, for every object $A$ in $\NN \cap \Set^{\A^\op}$ there is $B \in \NN$ such that $\RR(A) \cong B$, and similarly for morphisms.
\end{proposition}
\begin{proof}
By \thref{big-enough-n-for-subuniverses}, the relation $\RR(X) \cong Y$ is $\Sigma_n$-expressible from the parameters $\A$ and $\lambda$. So the formula $\exists Y(\RR(X) \cong Y)$ is $\Sigma_n$-expressible from the parameters $\A$ and $\lambda$, and the result follows.
\end{proof}
\begin{proposition}
\thlabel{ae-quasieffective-in-presheaves-iff-in-k}
Let $\A$ be a $\lambda$-presheaf presentation of $\K$ with reflection functor $\RR$, and let $n$ be as in \thref{big-enough-n-for-subuniverses}. Assume
\[
\A \cup \{ \A, \lambda \} \subset \NN \prec_n (V, \in).
\]
Then for any class of morphisms $\M$ in $\K$ the following are equivalent.
\begin{enumerate}[label=(\roman*)]
\item\label{item_M_quasi} For every $f \in \NN \cap \M$ the square below is $\M$-quasieffective.
\[\begin{tikzcd}[ampersand replacement=\&,cramped]
	A \& B \\
	{A \restriction^\K \NN} \& {B \restriction^\K \NN}
	\arrow["f", from=1-1, to=1-2]
	\arrow[hook, from=2-1, to=1-1]
	\arrow["{f \restriction^\K \NN}"', from=2-1, to=2-2]
	\arrow[hook, from=2-2, to=1-2]
\end{tikzcd}\]

\item\label{item_RpreimageM_quasi} For every $g \in \NN \cap \RR^{-1}(\M)$ the square below is $\RR^{-1}(\M)$-quasieffective.
\[\begin{tikzcd}
	A & B \\
	{A \restriction \NN} & {B \restriction \NN}
	\arrow["g", from=1-1, to=1-2]
	\arrow[hook, from=2-1, to=1-1]
	\arrow["{g \restriction \NN}"', from=2-1, to=2-2]
	\arrow[hook, from=2-2, to=1-2]
\end{tikzcd}\]
\end{enumerate}
\end{proposition}
\begin{proof}
We first prove \ref{item_M_quasi} $\Rightarrow$ \ref{item_RpreimageM_quasi}. So let $g \in \NN \cap \RR^{-1}(\M)$. By \thref{subuniverse-closed-under-reflection}, $\NN$ is closed under $\mathcal{R}$, so $\RR(g) \in \NN \cap \M$.\footnote{Technically, there is $g' \in \NN \cap \M$ such that $\RR(g) \cong g'$, but the argument that follows is insensitive to isomorphisms so we chose for notational simplification.}  Let $P$ be the pushout of $A \hookleftarrow A \restriction \NN \xrightarrow{g \restriction \NN} B \restriction \NN$ and let $u: P \to B$ be the induced morphism, as pictured in the diagram below on the left. Then $\RR$ sends this entire diagram to the one pictured on the right.
\[\begin{tikzcd}[ampersand replacement=\&,cramped,sep=small]
	A \&\& B \&\& {\RR(A)} \&\& {\RR(B)} \\
	\& P \&\& {\xmapsto{\quad \RR \quad}} \&\& {\RR(P)} \\
	{A \restriction \NN} \&\& {B \restriction \NN} \&\& \begin{array}{c} \RR(A \restriction \NN)\\= A \restriction^\K \NN \end{array} \&\& \begin{array}{c} \RR(B \restriction \NN)\\= B \restriction^\K \NN \end{array}
	\arrow["g", from=1-1, to=1-3]
	\arrow[curve={height=6pt}, from=1-1, to=2-2]
	\arrow["{\RR(g)}", from=1-5, to=1-7]
	\arrow[curve={height=6pt}, from=1-5, to=2-6]
	\arrow["u", from=2-2, to=1-3]
	\arrow["\lrcorner"{anchor=center, pos=0.125, rotate=-90}, draw=none, from=2-2, to=3-1]
	\arrow["{\RR(u)}", from=2-6, to=1-7]
	\arrow["\lrcorner"{anchor=center, pos=0.125, rotate=-90}, draw=none, from=2-6, to=3-5]
	\arrow[hook, from=3-1, to=1-1]
	\arrow["{g \restriction \NN}"', from=3-1, to=3-3]
	\arrow[hook, from=3-3, to=1-3]
	\arrow[curve={height=-6pt}, from=3-3, to=2-2]
	\arrow[hook, from=3-5, to=1-5]
	\arrow["\begin{array}{c} \RR(g \restriction \NN)\\= g \restriction^\K \NN \end{array}"', from=3-5, to=3-7]
	\arrow[hook, from=3-7, to=1-7]
	\arrow[curve={height=-6pt}, from=3-7, to=2-6]
\end{tikzcd}\]
By assumption \ref{item_M_quasi}, the outer square on the right side is $\M$-quasieffective. Furthermore, $\RR(P)$ is also the relevant pushout in that square because $\RR$ is left adjoint and thus preserves colimits. We conclude that $\RR(g \restriction \NN)$ and $\RR(u)$ are both in $\M$, as required.

We now prove the converse, \ref{item_RpreimageM_quasi} $\Rightarrow$ \ref{item_M_quasi}. So let $f \in \NN \cap \M$. Since $\K$ is a subcategory of $\Set^{\A^\op}$, we can view $f$ as a morphism in $\Set^{\A^\op}$. We can thus form the diagram in $\Set^{\A^\op}$ like below on the left, where $P$ is the pushout of the bottom left span and $u$ the induced morphism. Then $\RR$ sends the entire diagram to the one on the right, where we get $f$ at the top again because $\RR$ fixes $\K$.

\[\begin{tikzcd}[ampersand replacement=\&,cramped,sep=small]
	A \&\& B \&\& A \&\& B \\
	\& P \&\& {\xmapsto{\quad \RR \quad}} \&\& {\RR(P)} \\
	{A \restriction \NN} \&\& {B \restriction \NN} \&\& \begin{array}{c} \RR(A \restriction \NN)\\= A \restriction^\K \NN \end{array} \&\& \begin{array}{c} \RR(B \restriction \NN)\\= B \restriction^\K \NN \end{array}
	\arrow["f", from=1-1, to=1-3]
	\arrow[curve={height=6pt}, from=1-1, to=2-2]
	\arrow["f", from=1-5, to=1-7]
	\arrow[curve={height=6pt}, from=1-5, to=2-6]
	\arrow["u", from=2-2, to=1-3]
	\arrow["\lrcorner"{anchor=center, pos=0.125, rotate=-90}, draw=none, from=2-2, to=3-1]
	\arrow["{\RR(u)}", from=2-6, to=1-7]
	\arrow["\lrcorner"{anchor=center, pos=0.125, rotate=-90}, draw=none, from=2-6, to=3-5]
	\arrow[hook, from=3-1, to=1-1]
	\arrow["{f \restriction \NN}"', from=3-1, to=3-3]
	\arrow[hook, from=3-3, to=1-3]
	\arrow[curve={height=-6pt}, from=3-3, to=2-2]
	\arrow[hook, from=3-5, to=1-5]
	\arrow["\begin{array}{c} \RR(f \restriction \NN)\\= f \restriction^\K \NN \end{array}"', from=3-5, to=3-7]
	\arrow[hook, from=3-7, to=1-7]
	\arrow[curve={height=-6pt}, from=3-7, to=2-6]
\end{tikzcd}\]
We have $f \in \RR^{-1}(\M)$, so by assumption \ref{item_RpreimageM_quasi}, $f \restriction \NN$ and $u$ are in $\RR^{-1}(\M)$.  Since $\RR$ preserves colimits we have that $\RR(P)$ is the relevant pushout for the square on the right.  So $\RR(f \restriction \NN)$ (which equals $f \restriction^\K \NN$, by definition of $(-) \restriction^\K \NN$) together with the induced morphism $\RR(u)$ are both in $\M$, as required.
\end{proof}
\begin{proposition}
\thlabel{cellularly-generated-iff-preimage-cellularly-generated}
Let $F: \C \to \D$ be a colimit preserving functor between locally presentable categories that is surjective on morphisms and let $\M$ be a class of morphisms in $\D$. Then $\M$ is cellularly generated if and only if $F^{-1}(\M)$ is cellularly generated.
\end{proposition}
\begin{proof}
We first prove the right to left direction. Let $\X$ be a set such that $F^{-1}(\M) = \cell(\X)$. We claim that $\M = \cell(F(\X))$. Indeed, let $f \in \M$ then by assumption there is $g$ such that $F(g) = f$. So $g$ is a transfinite composition of pushouts of morphisms in $\X$ and as $F$ preserves colimits we have that $f$ is a transfinite compositions of pushouts of morphisms in $F(\X)$.

The right to left direction is the same as \cite[Remark 3.8]{makkaiCellularCategories2014}, which is based on \cite[Theorem 3.2]{makkaiCellularCategories2014}. The latter is stated for cofibrant generation, but its proof starts by reducing to cellular generation (cf.\ \thref{elimination-of-retracts}). The remainder of the proof works for cellular generation and so \cite[Theorem 3.2]{makkaiCellularCategories2014} holds for cellular generation.
\end{proof}

\begin{theorem}
\thlabel{cellular-generation-any-lp}
Let $\K$ be a locally $\lambda$-presentable category, let $\M$ be a cellularly closed class of morphisms in $\K$, and let $n$ be as in \thref{big-enough-n-for-subuniverses}. Then the following are equivalent.
\begin{enumerate}[label=(\roman*)]
\item The class $\M$ is cellularly generated.
\item The class $\M$ is almost everywhere quasieffective. That is, for any $\lambda$-presheaf presentation $\A$ of $\K$ with reflection functor $\RR$, there exists a regular uncountable $\kappa > |\A|$ and a parameter $p$ such that whenever
\[
\A \cup \{ p, \A, \kappa, \lambda \} \subset \NN \prec_n (V, \in)
\]
and $\NN \cap \kappa$ is transitive, the following holds. For every $f: A \to B$ in $\NN \cap \M$, the square below is $\M$-quasieffective.
\[\begin{tikzcd}[ampersand replacement=\&]
	A \& B \\
	{A \restriction^\K \NN} \& {B \restriction^\K \NN}
	\arrow["f", from=1-1, to=1-2]
	\arrow[hook, from=2-1, to=1-1]
	\arrow["{f \restriction^\K \NN}"', from=2-1, to=2-2]
	\arrow[hook, from=2-2, to=1-2]
\end{tikzcd}\]
\end{enumerate}
\end{theorem}
\begin{proof}
By \thref{cellularly-generated-iff-preimage-cellularly-generated} we have that $\M$ is cellularly generated if and only if $\RR^{-1}(\M)$ is cellularly generated, and by \thref{cellular-generation-lfp} we have that $\RR^{-1}(\M)$ cellularly generated if and only if $\RR^{-1}(\M)$ is almost everywhere quasieffective. We have thus reduced to showing that $\RR^{-1}(\M)$ is almost everywhere quasieffective if and only if $\M$ is almost everywhere quasieffective, which follows from \thref{ae-quasieffective-in-presheaves-iff-in-k}.
\end{proof}
\begin{remark}
\thlabel{parameter-in-cellular-generation-any-lp}
The parameter $p$ and cardinal $\kappa$ in \thref{cellular-generation-any-lp} are needed to apply \thref{cellular-generation-lfp}. So following that proof (or more precisely, referring to \thref{lfp-cellular-generation-implies-ae-quasieffective}) we see that we can take $p = \M_0$, where $\M_0$ is a set such that $\M = \cell(\M_0)$. Then $\kappa$ can be any regular uncountable cardinal such that $\kappa > |\M_0| + |\A|$ and every $f \in \M_0$ is $\kappa$-presentable in $\left( \Set^{\A^\op} \right)^2$. Indeed, the $T$ and $\L$ there can be constructed from $\A$ and $|\L| = |\A|$. Similar to our convention for the meaning of $\A \subset \NN$, the notation $|\A|$ is really shorthand for $|\Obj(\A) \cup \Mor(\A)|$.
\end{remark}
In \thref{cellular-generation-lfp} there was another equivalent condition (labelled (ii)), namely that every $f \in \M$ can be expressed as an appropriate filtration. One would be tempted to simply apply that theorem and use the reflection functor $\RR$ to reflect such filtrations to $\K$. Even though such a reflected filtration will still be a continuous chain, it may not be a filtration because $\RR$ need not preserve presentability ranks.
\begin{question}
\thlabel{locally-presentable-filtrations}
Is there an analogue of \thref{cellular-generation-lfp}(ii) for arbitrary locally presentable categories?
\end{question}

\section{Cellular generation for continuous classes of morphisms}
\label{sec:cellular-generation-for-continuous-classes-of-morphisms}
In this section we show how assuming that $\M$ is continuous (see \thref{continuous}) yields a characterisation of cofibrant/cellular generation in terms of accessibility of $\Eff[\M](\K)$, the category of effective squares (see \thref{effective-square}). Continuity is not automatic (see \thref{FreeQuotient}), but it is a natural assumption.
\begin{definition}
\thlabel{continuous}
A composable class of morphisms $\M$ in a category $\K$ with directed colimits is called \emph{continuous} if $\K_\M$ is closed under directed colimits in $\K$.
\end{definition}
In this section we will be interested in $\M$-quasieffective squares where all morphisms are in $\M$, so we make the following definition.\footnote{Our presentation of quasieffectiveness and effectiveness is dual to the actual history. The notion of effective square appeared first (called \emph{cellular} in \cite[Definition 2.2]{liebermanCellularCategoriesStable2023}), based on Barr's notion of effective unions \cite{barrCategoriesEffectiveUnions1988}. It is only in this paper that we introduced quasieffectiveness as a weakening.}
\begin{definition}
\thlabel{effective-square}
Let $\M$ be a class of morphisms in a cocomplete category $\K$. A commutative square of morphisms in $\M$ is called \emph{$\M$-effective} if the morphism from the relevant pushout is in $\M$.
\[\begin{tikzcd}
	A && B &&&& B \\
	&&& \Longleftrightarrow & A & P \\
	C && D && C & D
	\arrow["{\in \M}"{description}, from=1-1, to=1-3]
	\arrow[curve={height=-12pt}, from=2-5, to=1-7]
	\arrow[from=2-5, to=2-6]
	\arrow["{\in \M}"{description}, dashed, from=2-6, to=1-7]
	\arrow["\lrcorner"{anchor=center, pos=0.125, rotate=-90}, draw=none, from=2-6, to=3-5]
	\arrow["{\in \M}"{description}, from=3-1, to=1-1]
	\arrow["{\M\text{-effective}}"{description}, draw=none, from=3-1, to=1-3]
	\arrow["{\in \M}"{description}, from=3-1, to=3-3]
	\arrow["{\in \M}"{description}, from=3-3, to=1-3]
	\arrow[from=3-5, to=2-5]
	\arrow[from=3-5, to=3-6]
	\arrow[curve={height=12pt}, from=3-6, to=1-7]
	\arrow[from=3-6, to=2-6]
\end{tikzcd}\]
We write $\Eff[\M](\K)$ for the wide subcategory of $(\K_\M)^2$ on the $\M$-effective squares. That is, the objects of $\Eff[\M](\K)$ are the morphisms in $\M$ and the morphisms in $\Eff[\M](\K)$ are $\M$-effective squares.
\end{definition}
Any square where one of $C \to A$ or $C \to D$ is an isomorphism is $\M$-effective and by \cite[Theorem 2.7]{liebermanCellularCategoriesStable2023} $\M$-effective squares can be composed, so $\Eff[\M](\K)$ is indeed a subcategory of $(\K_\M)^2$.

There is a surprising connection to model theory here. The $\M$-effective squares form a stable independence relation, in the sense of \cite{liebermanForkingIndependenceCategorical2019}, if and only if $\Eff[\M](\K)$ is accessible. Stable independence is a central topic in model theory, specifically in stability theory, and generalises many known independence relations in mathematics (e.g., linear independence, algebraic independence, and probabilistic independence). The connection to cofibrant generation is given in \cite[Theorem 3.1]{liebermanCellularCategoriesStable2023}, accessibility of $\Eff[\M](\K)$ (and hence model-theoretic stability) is linked to cofibrant generation of $\M$. We improve on this theorem by also providing a version for cellular generation (\thref{cellular-generation-continuous}). We also fix a gap in its proof, see \thref{the-gap}.
\begin{theorem}
\thlabel{cofibrant-generation-continuous}
Let $\K$ be a locally presentable category and let $\M$ be a coherent continuous cofibrantly closed class of morphisms in $\K$. Then the following are equivalent.
\begin{enumerate}[label=(\roman*)]
\item The class $\M$ is cofibrantly generated and $\K_\M$ is accessible (with directed colimits).
\item The category $\Eff[\M](\K)$ is accessible (with directed colimits).
\end{enumerate}
\end{theorem}
\begin{theorem}
\thlabel{cellular-generation-continuous}
Let $\K$ be a locally presentable category and let $\M$ be a coherent continuous cellularly closed class of monomorphisms in $\K$. Then the following are equivalent.
\begin{enumerate}[label=(\roman*)]
\item The class $\M$ is cellularly generated and $\K_\M$ is accessible (with directed colimits).
\item The category $\Eff[\M](\K)$ is accessible (with directed colimits)
\end{enumerate}
\end{theorem}
\begin{proof}[Proof of \thref{cofibrant-generation-continuous,cellular-generation-continuous}]
We break up the proofs in a few parts, which are proved separately below, and where we are more precise about the exact assumptions that are needed for each direction. For both theorems, (i) $\Rightarrow$  (ii) follows from \thref{cellular-generation-implies-accessible} (in the case of \thref{cofibrant-generation-continuous} because cofibrantly generated classes are in particular cellularly generated by \thref{elimination-of-retracts}). The converse, (ii) $\Rightarrow$ (i), follows for each theorem from two out of \thref{effective-accessible-implies-cofibrant-generation,monos-effective-accessible-implies-cellular-generation,effective-accessible-implies-km-accessible}.
\end{proof}
As a corollary of the proof of \thref{cellular-generation-continuous} we again get that morphisms of such classes $\M$ can be expressed as transfinite compositions of cardinal length, just like in \thref{transfinite-compositions-of-cardinal-length}. However, in this case the proof is purely categorical and the category can be any locally presentable category.
\begin{corollary}
\thlabel{transfinite-compositions-of-cardinal-length-continuous}
Let $\M$ be a coherent continuous cellularly closed class of monomorphisms in a locally presentable category $\K$. Then there is $\mu$ such that for all regular $\kappa \geq \mu$, any morphism $f \in \M$ between $\kappa$-presentable objects can be expressed as a transfinite composition of length exactly $\kappa'$ of morphisms in $\Po(\M_\mu)$, for some (possibly finite) cardinal $\kappa' < \kappa$.
\end{corollary}
\begin{proof}
This is the ``moreover'' part of \thref{monos-effective-accessible-implies-cellular-generation}.
\end{proof}

The main point about continuity is that directed colimits behave nicely in $(\K_\M)^2$ and $\Eff[\M](\K)$. The former follows straightforwardly from contuinity and the latter is \cite[Lemma 2.11]{liebermanCellularCategoriesStable2023} (cellularity is not used).
\begin{fact}
\thlabel{continuous-gives-directed-colimits-of-effective-squares}
Let $\M$ be a continuous class of morphisms in a locally presentable category $\K$. Then categories $\K_\M^2$ and $\Eff[\M](\K)$ both have directed colimits and these are computed as in $\K^2$, so the inclusion functors $\Eff[\M](\K) \hookrightarrow (\K_\M)^2 \hookrightarrow \K^2$ both preserve directed colimits.
\end{fact}
\begin{lemma}
\thlabel{coherent-m-facts}
Let $\M$ be a coherent class of morphisms in a locally presentable category $\K$.
\begin{enumerate}[label=(\roman*)]
\item The commutative squares that have all morphisms in $\M$ form a coherent class of morphisms in $\K^2$.
\item If $\M$ is closed under pushouts then the $\M$-effective squares form a left-cancellable class of morphisms in $(\K_\M)^2$.
\item If $\M$ is continuous then for any regular cardinal $\lambda$, an object $A$ being $\lambda$-presentable in $\K$ implies that it is also $\lambda$-presentable in $\K_\M$.
\item If $\M$ is continuous and $\K_\M$ is accessible then there is some cardinal $\lambda$ such that $\K$ and $\K_\M$ agree on presentability ranks above $\lambda$.
\end{enumerate}
\end{lemma}
\begin{proof}
The first three items are straightforward. For (iv) we use \cite[Propositions 3.7 and 4.3]{bekeAbstractElementaryClasses2012}, the combination of which gives us exactly what we need. These propositions apply to any directed colimit preserving functor between accessible categories with all directed colimits, that reflect split epimorphisms. That last assumption is quickly verified: let $e \in \M$ be a split epimorphism in $\K$, then there is $m$ such that $em = id$, which implies $m \in \M$ by coherence.
\end{proof}
\begin{definition}
\thlabel{good-poset}
Let $\lambda$ be a cardinal. A poset $P$ is called \emph{$\lambda$-good} if it is well-founded, has a least element $\bot$ and for any $x \in P$ we have $|\{y \in P : y \leq x\}| < \lambda$.

An element $x \in P$ is called \emph{isolated} if $\{y \in P : y < x\}$ has a top element, which we will denote by $x^-$. 
\end{definition}
\begin{definition}
\thlabel{good-diagram}
Let $P$ be a $\lambda$-good poset and let $(A_x)_{x \in P}$ be a diagram. We call the diagram \emph{$\lambda$-good} if it is smooth. That is, for every non-isolated $x \in P$, with $x \neq \bot$, we have that $A_x = \colim_{y < x} A_y$.

The \emph{links} of such a $\lambda$-good diagram are the morphisms $A_{x^-} \to A_x$, where $x$ is isolated. The \emph{composite} of the diagram is the coprojection $A_\bot \to \colim_{x \in P} A_x$.
\end{definition}
The point of $\lambda$-good posets is that we can use them to replace the transfinite compositions in cellularly generated classes by $\lambda$-good $\lambda$-directed posets.
\begin{fact}[fat small object argument]
\thlabel{fat-small-object-argument}
Let $\K$ be a cocomplete category and let $\M_0$ be a class of morphisms in $\K$.
\begin{enumerate}[label=(\roman*)]
\item Suppose that there is $\lambda$ such that every morphism in $\M_0$ has $\lambda$-presentable domain. Then every morphism in $\cell(\M_0)$ is the composite of a $\lambda$-good $\lambda$-directed diagram with links in $\Po(\M_0)$.
\item For any $\lambda$, the composite of a $\lambda$-good diagram in $\K$ with links in $\Po(\M_0)$ is in $\cell(\M_0)$. Moreover, if $D: P \to \K$ is a $\lambda$-good diagram then its composite is a transfinite composition of length $< |P|$ of morphisms in $\Po(\M_0)$.
\end{enumerate}
\end{fact}
\begin{proof}
Item (i) is \cite[Theorem 4.11]{makkaiFatSmallObject2014}, and item (ii) is \cite[Proposition 4.5 and Remark 4.6]{makkaiFatSmallObject2014}.
\end{proof}
\begin{lemma}
\thlabel{iterating-lambda-directed-colimits}
Let $\K$ be any category with $\lambda$-directed colimits and let $\A$ be a class of $\lambda$-presentable objects in $\K$. Suppose that $B$ is the $\lambda$-directed colimit of a diagram $(B_i)_{i \in I}$ in $\K$, where $B_i$ is a $\lambda$-directed colimit of objects in $\A$, for each $i \in I$. Then $B$ is a $\lambda$-directed colimit of objects in $\A$.
\end{lemma}
\begin{proof}
Write $\lambda\text{-}\Ind(\A)$ for the free completion of $\A$, viewed as a full subcategory of $\K$, under $\lambda$-directed colimits. This construction, and the two facts below, are standard. The case where $\lambda = \omega$, can be found in \cite[Corollaire 8.5.2 and Proposition 8.7.5(a)]{grothendieckTheorieToposCohomologie1972}, or in \cite[Theorem 6.1.8, Corollary 6.3.2 and Proposition 6.3.4]{kashiwaraCategoriesSheaves2006} for a more modern presentation. Generalisation to uncountable $\lambda$ is straightforward.
\begin{enumerate}
\item We can view $\A$ as a full subcategory of $\lambda\text{-}\Ind(\A)$, and any object in the latter is a $\lambda$-directed colimit of objects in $\A$. Furthermore, $\lambda\text{-}\Ind(\A)$ has all $\lambda$-directed colimits.
\item Given an object $X$ in $\lambda\text{-}\Ind(\A)$ we can, by (1), write it as $X = \colim_{j \in J} A_j$ for some $\lambda$-directed $J$. Then the functor $F: \lambda\text{-}\Ind(\A) \to \K$ defined by $F(X) = \colim_{j \in J} A_j$ fixes $\A$, preserves $\lambda$-directed colimits, and is fully faithful.
\end{enumerate}
Let the notation be as in the statement of the lemma. For a fixed $i \in I$ we can compute the colimit of the diagram of $B_i$ in $\lambda\text{-}\Ind(\A)$ to find $C_i$ such that $F(C_i) = B_i$. By (2) we thus find a diagram $(C_i)_{i \in I}$ whose image under $F$ is exactly $(B_i)_{i \in I}$. Set $C = \colim_{i \in I} C_i$ so $F(C) = B$. Then $C$ is a $\lambda$-directed colimit of objects in $\A$ by (1), and so the same is true for $B$.
\end{proof}
\begin{remark}
\thlabel{iterating-lambda-directed-colimits-elementary-proof}
A more elementary, but lengthy, proof of \thref{iterating-lambda-directed-colimits} is also possible. In that case the assumption that $\K$ has all $\lambda$-directed colimits can be removed. However, we have no use for this slight generalisation, so we chose for the shorter proof above.
\end{remark}
\begin{proposition}
\thlabel{cellular-generation-implies-accessible}
Let $\M$ be a coherent continuous cellularly generated class of morphisms in a locally presentable category $\K$, such that $\K_\M$ is accessible. Then $\Eff[\M](\K)$ is accessible.
\end{proposition}
\begin{proof}
By \thref{coherent-m-facts}(iv) there is a cardinal $\lambda'$ such that $\K$ and $\K_\M$ agree on presentability ranks above $\lambda'$. By \thref{coherent-m-facts} again, items (i) and (iv), we may assume the same for $\K^2$ and $(\K_\M)^2$. We will only consider $\lambda$-presentable objects for some $\lambda \geq \lambda'$ in this proof, so when we say ``$\lambda$-presentable'' it does not matter in which category we consider the object.

By \thref{continuous-gives-directed-colimits-of-effective-squares} the category $\Eff[\M](\K)$ has directed colimits. Let $\M_0 \subseteq \M$ be a set such that $\M = \cell(\M_0)$. Let $\lambda \geq \lambda'$ be such that both $\K$ and $\K_\M$ are $\lambda$-accessible and every morphism in $\M_0$ is $\lambda$-presentable. Let $\M_\lambda$ be a representative set of all those morphisms in $\M$ that are $\lambda$-presentable, so $\M_0 \subseteq \M_\lambda$ and thus $\M = \cell(\M_\lambda)$.

By \thref{coherent-m-facts}, items (ii) and (iii), every object in $\M_\lambda$ is also $\lambda$-presentable in $\Eff[\M](\K)$. It thus suffices to show that every morphism in $\M$ is a $\lambda$-directed colimit in $\Eff[\M](\K)$ of morphisms in $\M_\lambda$.

Let $\M^*$ be the class of morphisms in $\M$ that are a $\lambda$-directed colimit in $\Eff[\M](\K)$ of morphisms in $\M_\lambda$. We show that $\cell(\M_\lambda) \subseteq \M^*$, because then $\M \subseteq \M^*$ and we have equality because $\M^* \subseteq \M$ by definition.

Firstly, let $f \in \M_\lambda$ and let $g$ be a pushout of $f$, as in the diagram below.
\[\begin{tikzcd}
	C & D \\
	A & B
	\arrow["g", from=1-1, to=1-2]
	\arrow["\lrcorner"{anchor=center, pos=0.125, rotate=-90}, draw=none, from=1-2, to=2-1]
	\arrow["a", from=2-1, to=1-1]
	\arrow["f"', from=2-1, to=2-2]
	\arrow["b"', from=2-2, to=1-2]
\end{tikzcd}\]
Since $\K_\M$ is $\lambda$-accessible, $C$ is a $\lambda$-directed colimit in $\K_\M$ of a diagram $(C_i)_{i \in I}$ of $\lambda$-presentable objects. By \thref{presentability-of-arrows} we have that $A$ is $\lambda$-presentable in $\K$. So $a: A \to C$ factors through some $C_i$. By restricting the diagram we may thus assume that $a$ factors through every object in the diagram. For each $i \in I$ we let $g_i: C_i \to D_i$ be the pushout of $f$ as in the diagram below.
\[\begin{tikzcd}
	C & D \\
	{C_i} & {D_i} \\
	A & B
	\arrow["g", from=1-1, to=1-2]
	\arrow["\lrcorner"{anchor=center, pos=0.125, rotate=-90}, draw=none, from=1-2, to=2-1]
	\arrow["{a_i}", from=2-1, to=1-1]
	\arrow["{g_i}"{description}, from=2-1, to=2-2]
	\arrow["{b_i}"', from=2-2, to=1-2]
	\arrow["\lrcorner"{anchor=center, pos=0.125, rotate=-90}, draw=none, from=2-2, to=3-1]
	\arrow[from=3-1, to=2-1]
	\arrow["f"', from=3-1, to=3-2]
	\arrow[from=3-2, to=2-2]
\end{tikzcd}\]
Then $g_i \in \M_\lambda$ because it is the pushout of an $\M$ and $D_i$ is the pushout of $\lambda$-presentables, and hence $\lambda$-presentable. By construction, $a_i \in \M$ and so $b_i \in \M$ as it is the pushout of $a_i$. Since colimits commute, we have that $g = \colim_{i \in I} g_i$ and so $g \in \M^*$, as required.

Now let $f \in \cell(\M_\lambda)$. By the fat small object argument, \thref{fat-small-object-argument}(i), there is a $\lambda$-good $\lambda$-directed diagram $(A_x)_{x \in P}$ with links in $\Po(\M_\lambda)$, such that $f$ is its composite. For $x \leq y$ in $P$ we write $f_{xy}: A_x \to A_y$ for the corresponding morphism in the diagram. For all such $x \leq y$, the following square is $\M$-quasieffective.
\[\begin{tikzcd}
	{A_\bot} & {A_y} \\
	{A_\bot} & {A_x}
	\arrow["{f_{\bot,y}}", from=1-1, to=1-2]
	\arrow[equal,from=2-1, to=1-1]
	\arrow["{f_{\bot,x}}"', from=2-1, to=2-2]
	\arrow["{f_{xy}}"', from=2-2, to=1-2]
\end{tikzcd}\]
These squares make $(f_{\bot,x})_{x \in P}$ into a $\lambda$-directed diagram in $\Eff[\M](\K)$, whose colimit is $f$. For every $x \in P$ the restriction $(A_y)_{y \leq x}$ is still a $\lambda$-good diagram. So by \thref{fat-small-object-argument}(ii) its composite $f_{\bot,x}$ is a tansfinite composition of morphisms in $\Po(\M_\lambda)$. Moreover, because $|\{y \in P : y \leq x\}| < \lambda$ we may assume this transfinite composition to be of length $< \lambda$. Then \cite[Lemma 4.20]{makkaiFatSmallObject2014} tells us precisely that that $f_{\bot,x} \in \Po(\M_\lambda)$. Thus $f_{\bot,x} \in \M^*$ for all $x \in P$, so $f$ is a $\lambda$-directed colimit of morphisms in $\M^*$. We conclude that $f \in \M^*$, because $\M^*$ is closed under $\lambda$-directed colimits in $\Eff[\M](\K)$ by \thref{iterating-lambda-directed-colimits}.
\end{proof}
\begin{remark}
\thlabel{the-gap}
The proof of \thref{cellular-generation-implies-accessible} is based on the similar direction in \cite[Theorem 3.1]{liebermanCellularCategoriesStable2023}, but the latter contains a gap. In both cases we define $\M^*$ as in \thref{cellular-generation-implies-accessible} and show that it is closed under pushouts. The gap is in showing closure under transfinite composition, which is done inductively on the length of the chain in \cite{liebermanCellularCategoriesStable2023}. The problem is that at limit stages we generally only get directed colimits and not $\lambda$-directed colimits. So the proof is valid for $\lambda = \omega$. For uncountable $\lambda$ we solved this issue here by using the fat small object argument.
\end{remark}
\begin{lemma}
\thlabel{effective-accessible-implies-cofibrant-generation}
Let $\M$ be a continuous cofibrantly closed class of morphisms in a locally presentable category $\K$. If $\Eff[\M](\K)$ is accessible then $\M$ is cofibrantly generated.
\end{lemma}
\begin{proof}
By \thref{continuous-gives-directed-colimits-of-effective-squares} $\Eff[\M](\K)$ has directed colimits and these are preserved by the inclusion functor into $\K^2$. We can thus apply \cite[Proposition 4.3]{bekeAbstractElementaryClasses2012}, which then exactly states that there is some $\mu$ such that for all $\kappa \geq \mu$ both $\Eff[\M](\K)$ and $\K^2$ are $\kappa$-accessible and the inclusion functor preserves $\kappa$-presentable objects. We claim that $\M = \cof(\M_\mu)$, for which we will show that $\M_\kappa \subseteq \cof(\M_\mu)$ by induction on $\kappa$. The base case, $\kappa = \mu$, is trivial and the limit stage follows because $\M_\kappa = \bigcup_{\kappa' < \kappa} M_{\kappa'}$ for limit $\kappa$ by \thref{successor-presentability-rank}.

Assume that $\M_\kappa \subseteq \cof(\M_\mu)$ and let $f \in \M$ be such that $\rank_{\K^2}(f) = \kappa^+$. We first argue that we may assume $\rank_{\Eff[\M](\K)}(f) = \kappa^+$. As $\Eff[\M](\K)$ is $\kappa^+$-accessible, there is a $\kappa^+$-directed diagram $(f_i)_{i \in I}$ of $\kappa^+$-presentables in $\Eff[\M](\K)$ such that $f = \colim_{i \in I} f_i$. Since the inclusion into $\K^2$ preserves directed colimits and $f$ is $\kappa^+$-presentable in $\K^2$, the identity on $f$ factors through this diagram in $\K^2$. So $f$ is a retract in $\K^2$ of $f_i$ for some $i \in I$. It thus suffices to show that $f_i \in \cof(\M_\mu)$. As the inclusion functor preserves $\kappa^+$-presentable objects, we have $\rank_{\K^2}(f_i) \leq \kappa^+$. If $\rank_{\K^2}(f_i) \leq \kappa$ then we are already done by the induction hypothesis. Otherwise $\rank_{\K^2}(f_i) = \kappa^+$ and we may replace $f$ by $f_i$.

Recall from \cite[Corollary 8.9(2)]{liebermanSizesFiltrationsAccessible2020} that in any $\kappa$-accessible category with directed colimits, any object $X$ with presentability rank at least $\kappa^+$ is a retract of an object $Y$ of the same presentability rank, and so that $Y$ admits a smooth filtration. Thus, working in $\Eff[\M](\K)$, we have that $f$ is a retract of some $f'$, which is the colimit of a smooth $\kappa$-small chain $(f_i)_{i < \delta}$. Considering this chain in $\K^2$, it is still a smooth $\kappa$-small chain because the inclusion functor preserves directed colimits and $\kappa$-presentable objects. By restricting to a cofinal subchain we may assume $\delta$ to be a cardinal, so $f' \in \cell(\M_\kappa) \subseteq \cof(\M_\mu)$ by \thref{build-cellular-generation-from-filtration}. As $f$ is a retract of $f'$ in $\Eff[\M](\K)$, it is in particular a retract of $f'$ in $\K^2$ and so $f \in \cof(\M_\mu)$, as required.
\end{proof}
\begin{lemma}
\thlabel{monos-effective-accessible-implies-cellular-generation}
Let $\M$ be a continuous cellularly closed class of monomorphisms in a locally presentable category $\K$. If $\Eff[\M](\K)$ is accessible then $\M$ is cellularly generated.

More precisely, there is a cardinal $\mu$ such that for all regular $\kappa \geq \mu$ any morphism $f \in \M$ between $\kappa$-presentable objects can be expressed as a transfinite composition of length exactly $\kappa'$ of morphisms in $\Po(\M_\mu)$, for some (possibly finite) cardinal $\kappa'$.
\end{lemma}
\begin{proof}
By \thref{continuous-gives-directed-colimits-of-effective-squares} $\Eff[\M](\K)$ has directed colimits and these are preserved by the inclusion functor into $\K^2$. Since $\M$ consists of monomorphisms, this inclusion functor also reflects split epimorphisms. By \cite[Propositions 3.7 and 4.3]{bekeAbstractElementaryClasses2012} there is thus a cardinal $\mu$ such that for all $\kappa \geq \mu$ both $\Eff[\M](\K)$ and $\K^2$ are $\kappa$-accessible and the inclusion functor preserves and reflects $\kappa$-presentable objects. Let $f \in \M$ be such that $\rank_{\K^2}(f) > \mu$. Recall from \cite[Corollary 8.9(3)]{liebermanSizesFiltrationsAccessible2020} that in any $\mu$-accessible category with directed colimits and all of whose morphisms are monomorphisms, any object of presentation rank at least $\mu^+$ admits a smooth filtration. In particular, there is a smooth filtration of $f$ in $\Eff[\M](\K)$. Since the inclusion functor preserves directed colimits and presentation ranks above $\mu$, it remains a smooth filtration of $f$ in $\K^2$. The result thus follows from \thref{filtrations-imply-cellular-generation}.
\end{proof}
\begin{lemma}
\thlabel{effective-accessible-implies-km-accessible}
Let $\M$ be a continuous composable class in $\K$ and suppose that $\Eff[\M](\K)$ is accessible. Then $\K_\M$ is accessible.
\end{lemma}
\begin{proof}
Define a functor $F: \K \to \K^2$ by sending $A$ to the identity morphism on $A$. We then have a pullback square of categories as below.
\[\begin{tikzcd}
	{\K_\M} & \K \\
	{\Eff[\M](\K)} & {\K^2}
	\arrow[hook, from=1-1, to=1-2]
	\arrow[from=1-1, to=2-1]
	\arrow["\lrcorner"{anchor=center, pos=0.125}, draw=none, from=1-1, to=2-2]
	\arrow["F", from=1-2, to=2-2]
	\arrow[hook, from=2-1, to=2-2]
\end{tikzcd}\]
Clearly, the bottom arrow is an isofibration, so the above square is a pseudopullback \cite[Theorem 1]{joyalPullbacksEquivalentPseudopullbacks1993}. By \thref{continuous-gives-directed-colimits-of-effective-squares} the inclusion at the bottom preserves directed colimits, and $F$ clearly preserves directed colimits. We conclude that $\K_\M$ is accessible, because accessible categories are closed under pseudopullbacks \cite[Exercise 2.n]{adamekLocallyPresentableAccessible1994}.
\end{proof}

We finish this section with an example of a cellularly generated class (in $\Ab$) that is not continuous, so \thref{cellular-generation-lfp} applies, but \thref{cellular-generation-continuous} does not. This shows that the former theorem really is more general.
\begin{theorem}
\thlabel{FreeQuotient}
In $\Ab$, let $\M$ denote the class of monomorphisms with free cokernel. Then $\M$ is cellularly generated by the single map $0 \to \Z$, but $\Eff[\M](\Ab)$ lacks directed colimits (in particular, $\M$ is not continuous in $\Ab$).
\end{theorem}
\begin{proof}
That $\M$ is cellularly generated by $0 \to \Z$ follows easily from the fact that every free abelian group is a direct sum of copies of $\Z$. 

Before showing that  $\Eff[\M](\Ab)$ lacks directed colimits, it is illustrative to first see why $\M$ is not continuous.  This construction appears in \cite[VII, Proposition 1.1]{MR1914985}. Consider the countable free abelian group $C$ with basis $\langle e_n : n \in \omega \rangle$, and for each $k \geq 0$ let $C_k = \langle e_0, e_1, \dots, e_k \rangle$, $u_k = e_k-2e_{k+1}$, and $A_k = \langle u_0, u_1, \dots, u_k \rangle$.  Note that $u_0,u_1,\dots, u_k$ are linearly independent and hence form a free basis of $A_k$. So each inclusion $A_k \subset A_{k+1}$ splits, and the quotient is free of rank $1$. For each $k$, the quotient $C/A_k$ is free, because the inclusion $A_k \to C_{k+1}$ splits, via a change-of-basis on the latter that sends $e_i$ to $u_i$ for $i \le k$, and fixes $e_{k+1}$. This change of basis matrix is invertible over the integers. So the chain $A_0 \subset A_1 \subset \dots$, together with their inclusion into $C$, form a cocone in $\Ab_\M$.  Let $A_\omega$ be the union (colimit) of the $A_k$'s. So the inclusion $A_\omega \to C$ is the unique map in $\Ab$ commuting with the cocone. But $C/A_\omega$ is not free, because in that quotient, $[e_0] = 2^n [e_n]$ for every $n \ge 0$ (and $[e_0]$ is nonzero in $C/A_\omega$). So it has a nonzero element that is divisible by $2^n$ for every $n > 0$.

We extract more out of this example to show $\Eff[\M](\Ab)$ lacks directed colimits.  Consider the commuting diagram
\[\begin{tikzcd}[sep=small]
	{A_0} & {A_1} & {A_k} & {A_{k+1}} & {A_\omega} \\
	0 & 0 & 0 & 0 & 0
	\arrow[from=1-1, to=1-2]
	\arrow[dashed, from=1-2, to=1-3]
	\arrow[from=1-3, to=1-4]
	\arrow[dashed, from=1-4, to=1-5]
	\arrow["{f_0}", from=2-1, to=1-1]
	\arrow[from=2-1, to=2-2]
	\arrow["{f_1}", from=2-2, to=1-2]
	\arrow[dashed, from=2-2, to=2-3]
	\arrow["{f_k}", from=2-3, to=1-3]
	\arrow[from=2-3, to=2-4]
	\arrow["{f_{k+1}}", from=2-4, to=1-4]
	\arrow[dashed, from=2-4, to=2-5]
	\arrow["f"', from=2-5, to=1-5]
\end{tikzcd}\]
where all maps are inclusions. Each square $f_k \to f_n$ (for $n \geq k$) is $\M$-effective, since its pushout is $A_k$, and $A_n/A_k = \langle u_{k+1},\dots u_n \rangle$ is free (and similarly for $f_k \to f$).  So $f$ is a cocone of $(f_k)_{k < \omega}$ in $\Eff[\M](\Ab)$.

Suppose towards a contradiction that $(f_k)_{k < \omega}$ has a colimit $g$ in $\Eff[\M](\Ab)$. Since $f$ is the colimit in $\Ab^2$ there is a morphism of cocones $f \to g$ in $\Ab^2$, and since $f$ forms a cocone of $(f_k)_{k < \omega}$ in $\Eff[\M](\Ab)$ there is also be a morphism of cocones $g \to f$. We thus have a morphism of cocones $f \to g \to f$, which by uniqueness must be $id_f$. So $g \to f$ is a retract (in $\Ab^2$), and it is also a monomorphism (because it comes from $\Eff[\M](\Ab)$), so it is an isomorphism and hence $g \cong f$. So $f$ must be the colimit of $(f_k)_{k < \omega}$ in $\Eff[\M](\Ab)$.

On the other hand, since $C/A_k$ is free for each $k < \omega$, the square below is $\M$-effective.
\[\begin{tikzcd}[sep=small]
	{A_k} & C \\
	0 & 0
	\arrow[from=1-1, to=1-2]
	\arrow["{f_k}", from=2-1, to=1-1]
	\arrow[from=2-1, to=2-2]
	\arrow["h"', from=2-2, to=1-2]
\end{tikzcd}\]
So $h$ is a cocone of $(f_k)_{k < \omega}$ in $\Eff[\M](\Ab)$. Then there would be an induced morphism $f \to h$ in $\Eff[\M](\Ab)$, which also must be the induced arrow in $\Ab^2$. This induced morphism must thus be given by the commuting square below, where $A_\omega \to C$ is the inclusion from earlier.
\[\begin{tikzcd}[sep=small]
	{A_\omega} & C \\
	0 & 0
	\arrow[from=1-1, to=1-2]
	\arrow["f", from=2-1, to=1-1]
	\arrow[from=2-1, to=2-2]
	\arrow["h"', from=2-2, to=1-2]
\end{tikzcd}\]
We also saw before that $C/A_\omega$ is not free, so the inclusion is not in $\M$ and we conclude that the square is not $\M$-effective, contradicting that it is a morphism in $\Eff[\M](\Ab)$. We also note that it is not even $\M$-quasieffective, as $A_\omega$ is also the relevant pushout for that square and so the induced morphism from the pushout into $C$ is the same problematic inclusion. 
\end{proof}

\appendix

\section{L\'evy complexities}
\label{sec_LevyComplexity}
\begin{lemma}
\thlabel{levy-complexity-colimits-mod-t}
Let $T$ be a limit theory in a finitary language $\L$, and let $\Phi(C,D,T,\L)$ denote the relation of $C$ being a colimit of the diagram $D$ in $\Mod(T)$. Then $\Phi$ is $\Sigma_1$-expressible in $(V,\in)$.
\end{lemma}
\begin{proof}
Let $\Str(\L)$ denote the category of $\L$-structures, and the morphisms are homomorphisms. A representative colimit of $D$ in $\Mod(T)$ can be constructed in two steps:
\begin{enumerate}[label=(\Alph*)]
    \item\label{item_colim_in_StrL} First construct the colimit of $D$ in $\Str(\L)$.
    \item\label{item_reflect_into_Mod_T} Then take the reflection of that $\L$-structure in $\Mod(T)$.
\end{enumerate}
So $\Phi(C,D,T,\L)$ holds if and only if $C$ is isomorphic to that particular representative; i.e., if and only if the following assertion holds:
\begin{align*}
\exists C_2 \ \exists C_1 \ \exists f \ \Big( & \underbrace{C_1  \text{ is the canonical } \Str(\L)\text{-colimit of } D }_{\Phi_1(C_1,\L, D)}  \\
& \wedge \  \underbrace{C_2 \text{ is the canonical reflection of } C_1 \text{ from } \Str(\L) \text{ into } \Mod(T)}_{\Phi_2(C_2,C_1,\L, T)} \\
& \wedge  \ \underbrace{ f \text{ is an isomorphism from } C \text{ to } C_2}_{\Phi_3(f,C,C_2,\L)}   \Big)
\end{align*}
It suffices to show that each of $\Phi_1$, $\Phi_2$, and $\Phi_3$ is $\Sigma_1$-expressible.  

The statement $\Phi_3(f,C,C_2,\L)$ is in fact $\Delta_0$-expressible, because it is the conjunction of ``$f$ is an $\L$-homomorphism from $C$ to $C_2$'' and ``$f$ is a bijection from $C$ to $C_2$'', each of which is expressible using only bounded quantifiers.  The assumption that $\L$ is a \emph{finitary} language is crucial for the $\Delta_0$-expressibility of ``$f$ is a homomorphism'', since it depends on the fact that if $A$ is a set, the set $A^{<\omega}$ is $\Delta_0$-definable in ZF from the parameter $A$ \cite[Chapter IV, Theorem 5.3]{MR597342}.

Next we show that $\Phi_1$ is $\Sigma_1$-expressible.  Let $\mathcal{A}$ be the (small) domain of the functor $D$. Then the colimit of $D$ in $\Str(\L)$ can be constructed as follows: 
\begin{enumerate}
    \item  Let $X$ be the disjoint union of the $\mathcal{L}$-structures indexed by $D$, i.e., 
\[
X= \{ (a,x) : a \in \Obj(\A) \text{ and } x \in D(a) \}.
\]
   This set is $\Delta_0$-definable from the parameter $D$, because $z \in X$ if and only if $\exists a \in \Obj(\A)$ and $\exists x \in D(a)$ such that $z=(a,x)$.  That last sentence still has some untranslated parts; e.g., since we view $\A=\dom(D)$ as the tuple $(\Obj(\A), \Mor(\A), \circ_\A, (id_a)_{a \in \Obj(\A)})$, the fragment ``$\exists a \in \Obj(\A)$ such that \dots" is short for ``there exists an $a$ that is an element of the first coordinate of $\dom(D)$ such that \dots".  That this and the other shorthands are also $\Delta_0$-expressible follows from the complexity computations in \cite[Chapter IV]{MR597342}, which we will implicitly use throughout this argument. 
   
   \item Let $\L_X$ augment $\L$ by adding a new constant symbol $d_{a,x}$ for every $(a,x) \in X$, and let $\Term(\L_X)$ be the set of $\L_X$-terms. There are various ways one could code up $\L_X$, but the key point is that we are working in a finitary language, so the terms are members of $\L_X^{<\omega}$, which is a $\Delta_0$-definable from the parameter $\L_X$ \cite[Chapter IV Theorem 5.3]{MR597342}.  This would \emph{not} be true if the language were infinitary.
   
   \item Let $R$ be the equivalence relation on $\Term(\L_X)$ generated by declaring $d_{a,x}$ equivalent to $d_{b,y}$ if there is a morphism $f: a \to b$ in $\A$ such that $D(f)(x)=y$, and  insisting that for every function symbol $h$ in $\L$ of arity $n$, if $t_i$ is equivalent to $t'_i$ for all $i \le n$ then $h(t_1,\dots,t_n)$ is equivalent to $h(t'_1, \dots, t'_n)$.  Then $R$ is a $\Delta_0$-definable subset of $\Term(\L_X)^2$ using the parameters $D$, $X$, and $\L$.

   \item Finally, the colimit of $D$ in $\Str(\L)$ has underlying set $\Term(\L_X) / R$ (which is $\Delta_0$-definable from $R$ and $\Term(\L_X)$), with the relation and function symbols from $\L$ interpreted in the minimal way to ensure that the maps $D(a) \to \Term(\L_X) / R$, which send elements to their $R$-equivalence classes, are homomorphisms.
\end{enumerate}

To summarise:  $\Phi_1(C_1,\L,D)$ can be expressed by
\begin{align*}
\exists X \  \exists \Gamma \  \exists R \   \Big( &X \text{ is the disjoint union of $D$-indexed $\L$-structures, } \\
&  \Gamma \text{ is the set of $\mathcal{L}_X$-terms,} \\
& R \text{ is the equivalence relation described above,} \\
& \text{and $C_1$ is the quotient $\Gamma/R$} \Big),
\end{align*}
where the part in parentheses is $\Sigma_1$-expressible (in fact, $\Delta_0$-expressible) in the parameters $D$ and $\mathcal{L}$.

Finally, we show why $\Phi_2(C_2,C_1,\mathcal{L},T)$---which says $C_2$ is the canonical reflection of $C_1$ from $\Str(\L)$ to $\Mod(T)$---is $\Sigma_1$-expressible.  Each sentence in $T$ is by assumption a limit sentence $\forall \bar{x} (\phi(\bar{x}) \to \exists! \bar{y} \psi(\bar{x}, \bar{y}))$, which is equivalent to the two sentences $\forall \bar{x} (\phi(\bar{x}) \to \exists \bar{y} \psi(\bar{x}, \bar{y}))$ and $\forall \bar{x} \bar{y} \bar{y}' (\phi(\bar{x}) \wedge \psi(\bar{x}, \bar{y}) \wedge \psi(\bar{x}, \bar{y}') \to \bar{y} = \bar{y}')$. Using this trick to get rid of ``unique'' part in the existential quantifiers, we may assume that $T$ consists of the formulas 
\[
\forall \overline{x}_i \left(\underbrace{ \varphi_i(\overline{x}_i) \to \exists \overline{y}_i \psi_i(\overline{x}_i,\overline{y}_i) }_{\tau_i(\overline{x}_i)} \right) 
\]
for all $i$ in some index set $I$, where each $\varphi_i$ and each $\psi_i$ is a conjunction of atomic formulas.

For an $\mathcal{L}$-structure $A$, define $B(A) \in \Str(\L)$ and $\pi(A): A \to B(A)$ as follows; this will be the procedure to be iterated in constructing a reflection of $A$.
\begin{enumerate}
    \item Let $\L'$ be the language $\L$ augmented with the following new constant symbols. For each $a$ in $A$ we introduce a symbol $c_a$. Furthermore, for each $i \in I$ and each tuple $\bar{a}$ in $A$ such that $A \models \phi_i(\bar{a})$ we introduce a tuple $\bar{d}_{i,\bar{a}}$ of constant symbols matching the length of $\bar{y}_i$. Let $X$ be this set of new constant symbols. Then $X$ is $\Delta_0$-definable from the parameters $A$, $T$, and $\L$.

    \item Let $E$ be the set of relations on $X$ consisting of the following.
    \begin{enumerate}
        \item The atomic diagram of $A$ (i.e., if $R$ is a relation symbol in $L$ and $A \models R(a_1,\dots,a_k)$ then $R(c_{a_1}, \dots, c_{a_k}) \in E$)
        \item For each $\bar{a} = (a_1,\dots,a_k)$ in $A$ and each $i \in I$ such that $A \models \varphi_i(\bar{a})$, $E$ includes each conjunct in the formula $\psi_i(c_{a_1},\dots,c_{a_k},\bar{d}_{i,\bar{a}})$ (recall that $\psi_i$ is a conjunction of atomic formulas)
    \end{enumerate}
    The set $E$ is $\Sigma_1$-definable from the parameters $A$, $\L$, and $T$.
    
    \item Let $B(A)$ be the free $\L$-structure on $X$ modulo the relations in $E$ and let $\pi(A): A \to B(A)$ be the factor map.  This is $\Sigma_1$-definable from $X$ and $E$. 
\end{enumerate}
The relation $\Theta(A,\pi,B)$ asserting that ``$B=B(A)$ and $\pi = \pi(A)$'' is $\Sigma_1$-expressible in $(V,\in)$ using the parameters $\L$ and $T$. By construction, $B(A)$ has the property that $\tau_i(\bar{b})$ holds for every $i \in I$ and every $\bar{b}$ of the appropriate length in the image of the map $\pi(A): A \to B(A)$.

These facts, together with the freeness of $B(A)$ modulo the relations, ensure that for an $\L$-structure $A$, a reflection $A_T$ of $A$ in $\Mod(T)$ can be constructed as a colimit (in $\Str(\L)$) of the chain $(A_n)_{n < \omega}$, where $A_0:=A$, $A_{n+1}:=B(A_n)$, and the link $A_n \to A_{n+1}$ is $\pi(A_n)$. The coprojection $\pi_0: A_0 \to A_T$ is a reflection of $A$ in $\Mod(T)$.  And $\Phi_2(C_2,C_1,\L,T)$ (i.e., the assertion that $C_2=(C_1)_T$) can be expressed by saying that there is a function $F$ with domain $\omega$ such that:
\begin{itemize}
    \item for all $n < \omega$, $F(n)$ is an ordered pair such that 
        \[
        \Theta\Big( (F(n))_0, (F(n))_1, (F(n+1))_0 \Big),
        \]
        and $(F(0))_0=C_1$;
    \item $C_2=\colim_{\Str(\L)} F$.
\end{itemize}
The first item is $\Sigma_1$-expressible because in ZF, if a relation $R(u,v,\dots)$ is $\Sigma_1$-expressible then so is the relation $\forall u \in v \ R(u,v,\dots)$; the proof uses the Collection Principle (see \cite[Lemma 13.10(i)]{MR1940513}). The second item is $\Sigma_1$-expressible by the earlier part of the proof that $\Phi_1$ was $\Sigma_1$-expressible; i.e., colimits in the category of $\L$-structures are $\Sigma_1$-definable.
\end{proof}

\bibliographystyle{alpha}
\bibliography{bibfile}

\end{document}